\documentclass[10pt]{amsart}

\usepackage{amsfonts,amssymb,amsmath,amscd,amstext}
\usepackage[colorlinks=true,linkcolor=blue,citecolor=blue]{hyperref}
\usepackage[utf8]{inputenc}
\usepackage{microtype}
\usepackage{graphicx}
\usepackage{changes}
\usepackage{comment}

\renewcommand{\le}{\leqslant}
\renewcommand{\ge}{\geqslant}
\newcommand{\ptl}{\partial}

\newcommand{\wdw}{\wedge \cdots \wedge}
\newcommand{\LMD}{\mathbf{\Lambda}}
\newcommand{\GT}{\mathbf{\tilde{G}}}
\newcommand{\rr}{{\mathbb{R}}}

\newcommand{\hh}{{\mathbb{H}}}
\newcommand{\nn}{{\mathbb{N}}}

\newcommand{\la}{\lambda}

\newcommand{\Om}{\Omega}
\newcommand{\eps}{\varepsilon}

\newcommand{\ga}{\gamma}
\newcommand{\Ga}{\Gamma}

\newcommand{\escpr}[1]{\langle#1\rangle}
\newcommand{\co}{\cos(\theta)}
\newcommand{\si}{\sin(\theta)}
\newcommand{\kk}{\kappa}
\newcommand{\mh}{\mathcal{H}}

\DeclareMathOperator{\divv}{div}

\newtheorem{theorem}{Theorem}[section]
\newtheorem{proposition}[theorem]{Proposition}
\newtheorem{lemma}[theorem]{Lemma}
\newtheorem{corollary}[theorem]{Corollary}

\theoremstyle{definition}

\newtheorem{remark}[theorem]{Remark}
\newtheorem{example}[theorem]{Example}

\newtheorem{definition}[theorem]{Definition} 

\theoremstyle{remark}

\numberwithin{equation}{section}

\definechangesauthor[name=Gianmarco, color=blue]{G}
\definechangesauthor[name=Manuel, color=purple]{M}
\definechangesauthor[name=Giovanna, color=violet]{C}

\begin{document}

\title{Variational formulas for submanifolds of~fixed~degree}

\author[G.~Citti]{Giovanna Citti}
\address{Dipartimento di Matematica, Piazza di Porta S. Donato 5, 401
26 Bologna, Italy}
\email{giovanna.citti@unibo.it}

\author[G.~Giovannardi]{Gianmarco Giovannardi}
\address{Departamento de Geometr\'{\i}a y Topolog\'{\i}a \& Research Unit MNat \\
Universidad de Granada \\ Granada \\ Spain}
\email{giovannardi@ugr.es}

\author[M.~Ritor\'e]{Manuel Ritor\'e} 
\address{Departamento de Geometr\'{\i}a y Topolog\'{\i}a \& Research Unit MNat \\
Universidad de Granada \\ Granada \\ Spain}
\email{ritore@ugr.es}

\date{\today}

\thanks{The authors have been supported by Horizon 2020 Project ref. 777822: GHAIA,  MEC-Feder grant MTM2017-84851-C2-1-P and PRIN 2015 ``Variational and perturbative aspects of nonlinear differential problems” (GC and GG)}

\thispagestyle{empty}
\bibliographystyle{abbrv} 

\begin{abstract}
We consider in this paper an area functional defined on  submanifolds of fixed degree immersed into a graded manifold equipped with a Riemannian metric. Since the expression of this  area depends on the degree, not all variations are admissible. It turns out that the associated variational vector fields must satisfy a system of partial differential equations of first order on the submanifold. Moreover,  given a vector field solution of this system, we provide a sufficient condition that guarantees the possibility of deforming the original submanifold by variations preserving its degree. As in the case of singular curves in sub-Riemannian geometry, there are examples of isolated surfaces that cannot be deformed in any direction. When the deformability condition holds we compute the Euler-Lagrange equations. The resulting mean curvature operator can be of third order. 
\end{abstract}

\maketitle

\thispagestyle{empty}

\tableofcontents

\section{Introduction}

The aim of this paper is to study the critical points of an area functional for submanifolds of given degree immersed in an equiregular graded manifold. This can be defined as the structure $(N,\mh^1,\ldots,\mh^s)$, where $N$ is a smooth manifold and $\mh^1\subset\mh^2\subset\cdots\subset\mh^s=TN$ is a flag of sub-bundles of the tangent bundle satisfying $[\mh^i,\mh^j]\subset \mh^{i+j}$ when $i,j\ge 1$ and $i+j\le s$, and $[\mh^i,\mh^j]\subset \mh^s$ when $i,j\ge 1$ and $i+j>s$. The considered area depends on the degree of the submanifold. The concept of pointwise degree for a submanifold $M$ immersed in a graded manifold was first introduced by Gromov in \cite{Gromov} as the homogeneous dimension of the tangent flag given by
\[
 T_p M \cap \mh_p^1 \subset \cdots \subset T_p M \cap \mh_p^s.
\]
The \emph{degree of a submanifold} $\deg(M)$ is the maximum of the pointwise degree among all points in $M$. An alternative way of defining the degree is the following: on an open neighborhood of a point $p \in N$ we can always consider a local basis $(X_1,\ldots,X_n)$ adapted to the filtration $(\mh^i)_{i=1,\ldots,s}$, so that each $X_j$ has a well defined degree. Following \cite{vittonemagnani} the degree of a simple $m$-vector $X_{j_1}\wedge \ldots \wedge X_{j_m}$ is the sum of the degree of the vector fields of the adapted basis appearing in the wedge product. Since we can write a $m$-vector tangent to $M$ with respect to the simple $m$-vectors of the adapted basis, the \emph{pointwise degree}  is given by the maximum of the degree of these simple $m$-vectors. 

We consider a Riemannian metric $g=\escpr{\cdot,\cdot}$ on $N$. 
For any $p\in N$, we get an orthogonal decomposition $T_pN=\mathcal{K}_p^1\oplus\ldots\oplus\mathcal{K}_p^s$. Then we apply to $g$ a dilation induced by the grading, which means that, for any $r>0$, we take the Riemannian metric $g_r$ 
making the subspaces $\mathcal{K}_p^i$ orthogonal and such that
\[
g_r|_{\mathcal{K}^i}=\frac{1}{r^{i-1}}\,g|_{\mathcal{K}^i} \; .
\]
 Whenever $\mh^1$ is a bracket generating distribution the structure $(N, g_r)$ converges in Gromov-Hausdorff sense to the sub-Riemannian structure $(N,\mh^1, g_{|\mh^1})$ as $r \to 0$. Therefore an immersed submanifold $M\subset N$ of degree $d$ has Riemannian area measure $A(M,g_r)$ with respect to the metric $g_r$. We define area measure $A_d$ of degree $d$ by
\begin{equation}
\label{intro:d area}
A_d(M):=\lim_{r\downarrow 0 }\ r^{(\deg(M)- \dim(M))/2} A(M,g_r)
\end{equation}
when this limit exists and it is finite. In \eqref{eq:projected_integral_formula_Ad} we stress that the area measure $A_d$ of degree $d$ is given by integral of the norm the $g$-orthogonal projection onto the subspace of $m$-forms of degree equal to $d$ of the orthonormal $m$-vector tangent to $M$. This area formula was provided in \cite{vittonemagnani, MTV} for $C^1$ submanifolds immersed in Carnot groups and in \cite{FSS07} for intrinsic regular submanifolds in the Heisenberg groups.

Given a submanifold $M\subset N$ of degree $d$ immersed into a graded manifold $(N,(\mh^i)_{i})$, we wish to compute the Euler-Lagrange equations for the area functional $A_d$. The problem has been intensively studied for hypersurfaces, and results appeared in \cite{GarofaloNh, DanielliGarofalo, CJHHF, ChengMalchiodi, DGN09,BASCV, HladPauls,HP13,HRR,RitoreRosales1,R09, Montefalcone, CS06}. For submanifolds of codimension greater than one in a sub-Riemannian structure only in the case of curves has been studied.  In particular it is well know that there exists minimizers of the length functional which are not solutions of the geodesic equation: these curves, discovered by Montgomery in \cite{Mont94a,Mont94b} are called abnormal geodesics. In this paper we recognize that a similar phenomenon can arise while studying the first variational of area for surfaces immersed in a graded structure: there are isolated surfaces which does not admit degree preserving variations. Consequently we focus on smooth submanifolds of fixed degree, and \emph{admissible} variations, which preserve it. The associated \emph{admissible vector fields}, $V= \frac{\partial \Gamma_t}{\partial t }\big|_{t=0}$ satisfies the system of partial differential equations of first order \eqref{System of PDEs for admissible} on $M$. So we are led to the central question of characterizing the admissible vector fields which are associated to an admissible variation. 

The analogous integrability problem for geodesics in sub-Riemannian manifolds and, more generally, for functionals whose domain of definition consists of integral curves of an exterior differential system,  was posed by E.~Cartan \cite{Cartan} and studied by P. Griffiths \cite{MR684663}, R. Bryant \cite{MR924805} and L. Hsu \cite{MR1189496}. These one-dimensional problems have been treated by considering a holonomy map \cite{MR1189496} whose surjectivity defines a \emph{regularity} condition implying that any vector field satisfying the system \eqref{System of PDEs for admissible} is integrable.
In higher dimensions, there does not seem to be an acceptable generalization of such an holonomy map. However, an analysis of Hsu's regularity condition led the authors to introduce a weaker condition named \emph{strong regularity} in \cite{1902.04015}. This condition can be generalized to higher dimensions and provides a sufficient condition to ensure the local integrability of any admissible vector field  on $M$, see Theorem~\ref{teor:locint}. 
Indeed, in this setting the 
admissibility system \eqref{System of PDEs for admissible} in coordinates is given by
\begin{equation}
\sum_{j=1}^m C_j(\bar{p}) E_j(F)(\bar{p})+ B(\bar{p})F(\bar{p})+ A(\bar{p})G(\bar{p})=0,
\label{intro:local system of PDEs3}
\end{equation}
where $C_j, B, A$ are matrices, $F$ are the vertical components of the admissible vector field, $G$ are  the horizontal control components and $\bar{p} \in M$.
Since the strong regularity tells us that  the matrix $A(\bar{p})$ has  full rank we can locally write explicitly a part of the controls in terms of the vertical components and the other part of the controls, then applying the Implicit Function Theorem we produce admissible variations.

In Remark \ref{rk:regularGP} we recognize that our definition of strongly regular immersion generalizes the notion introduced by \cite{Gromov} of regular horizontal immersions, that are submanifolds  immersed in the horizontal distribution such that the degree coincides with the topological dimension $m$.  In \cite{Gromov86}, see also \cite{Pansu16}, the author shows a deformability theorem for regular horizontal immersions  by means of Nash's Implicit Function Theorem \cite{N56}. Our result is in the same spirit but for immersions of general degree.

 For strongly regular submanifolds it is possible to compute the Euler-Lagrange equations to obtain a sufficient condition for stationary points of the area $A_d$ of degree $d$. This naturally leads to a notion of mean curvature, which is not in general a second order differential operator, but can be of order three. This behavior doesn't show up in the one-dimensional case where the geodesic equations for regular curves have order less than or equal to two, see  \cite[Theorem 7.2]{1902.04015} or \cite[Theorem 10]{MR1189496}.

These tools can be applied to mathematical model of perception in the visual cortex: G. Citti and A. Sarti in \cite{CS06} showed that 2 dimensional  minimal surfaces in the three-dimensional sub-Riemannian manifold  $SE(2)$ play an important role in the completion process of images,  taking orientation into account. Adding curvature to the model,  a four dimensional Engel structure arises, 
see \S~1.5.1.4 in \cite{Petitot2014}, \cite{DobbinsZucker} and \S~\ref{sc:2jetspace} here.
The previous $2D$ surfaces, lifted  in this structure are codimension 2, degree four strongly regular surfaces in the sense of our definition. On the other hand we are able to show that there are isolated surfaces which do not admit degree preserving variations. Indeed, in Example~\ref{ex:isolated} we exhibit an isolated plane, immersed in the Engel group, whose only admissible normal vector field is the trivial one. Moreover, in analogy with the one-dimensional result by \cite{MR1240644}, Proposition~\ref{prop:rigid} shows that  this isolated plane is rigid in the $C^1$ topology, thus this plane is a local minimum for the area functional. Therefore we recognized that a similar phenomenon to the one of existence of abnormal curves can arise in higher dimension. Finally we conjecture that a bounded open set $\Omega$ contained in this isolated plane is a global minimum among all possible  immersed surfaces sharing the same boundary $\partial \Omega$. 

We have organized this paper into several sections. In the next one notation and basic concepts, such as graded manifolds, Carnot manifolds and degree of submanifolds, are introduced. In Section~\ref{sec:area} we define the area of degree $d$ for submanifolds of degree $d$ immersed in a graded manifold $(N,\mh^i)$ endowed with a Riemannian metric. This is done as a limit of Riemannian areas. In addition, an integral formula for this area in terms of a density is given in formula \eqref{eq:integral_formula_Ad}. Section~\ref{sc:examples} is devoted to provide examples of submanifolds of certain degrees and the associated area functionals. In Sections~\ref{sc:admissible} and \ref{sc:structure} we introduce the notions of admissible variations, admissible vector fields and integrable vector fields and we study the system of first order partial differential equations defining the admissibility of a vector field. In particular, we show the independence of the admissibility condition for vector fields of the Riemannian metric in \S~\ref{ssc:independence}. In Section~\ref{sc:integrability} we give the notion of a strongly regular submanifold of degree $d$, see Definition~\ref{def:regular}. Then we prove in Theorem~\ref{teor:locint} that the strong regularity condition implies that any admissible vector vector is integrable. In addition, we exhibit in Example~\ref{ex:isolated} an isolated plane whose only admissible normal vector field is the trivial one. Finally in Section~\ref{sc:first} we compute the Euler-Lagrange equations of a strongly regular submanifold and give some examples.

\section{Preliminaries}

Let $N$ be an $n$-dimensional smooth manifold. Given two smooth vector fields $X,Y$ on $N$, their \emph{commutator} or \emph{Lie bracket} is defined by $[X,Y]:=XY-YX$.  An \emph{increasing filtration} $(\mh^i)_{i\in \nn}$ of  the tangent bundle $TN$ is a flag of sub-bundles
\begin{equation}
\mathcal{H}^1\subset\mathcal{H}^2\subset\cdots\subset \mathcal{H}^i\subset\cdots\subseteq TN,
\label{manifold flag}
\end{equation}
such that
\begin{enumerate}
\label{def:incfilt}
\item[(i)] $ \cup_{i \in \nn} \mh^i= TN$ 
\item[(ii)]
 $ [\mathcal{H}^{i},\mathcal{H}^{j}] \subseteq \mathcal{H}^{i+j},$ for $ i,j \ge1$, 
\end{enumerate}
where  
$ [\mathcal{H}^i,\mathcal{H}^j]:=\{[X,Y]  : X \in \mathcal{H}^i,Y \in \mathcal{H}^j\}$.
Moreover, we say that an increasing filtration is \emph{locally finite} when
\begin{enumerate}
\item[(iii)]  for each $p \in N$ there exists an integer  $s=s(p)$, the \emph{step} at $p$, satisfying $\mathcal{H}^s_p=T_p N$. Then  we have the following flag of subspaces
\begin{equation}
 \mathcal{H}^1_p\subset\mathcal{H}^2_p\subset\cdots\subset \mathcal{H}^s_p=T_p N.
 \label{flag at each point}
\end{equation}
\end{enumerate}

A \textit{graded manifold} $(N,(\mh^i))$ is a smooth manifold $N$ endowed
with a locally finite increasing filtration, namely  a flag of sub-bundles \eqref{manifold flag} satisfying (i),(ii) and (iii). For the sake of brevity a locally finite increasing filtration will be simply called a filtration.
Setting $n_i(p):=\dim {\mathcal{H}}^i_p $, the integer list $(n_1(p),\cdots,n_s(p))$  is called the \textit{growth vector} of the filtration \eqref{manifold flag} at $p$. When the growth vector is constant in a neighborhood of a point $p \in N$ we say that $p$ is a \textit{regular point} for the filtration. We say that a filtration $(\mathcal{H}^i)$ on a manifold $N$ is \textit{equiregular} if the growth vector is constant in $N$. From now on we suppose that $N$ is an equiregular graded manifold.

Given a vector $v$ in $T_p N$ we say that the \textit{degree} of $v$ is equal to $\ell$ if $v\in\mathcal{H}_p^\ell$ and $v \notin \mathcal{H}_p^{\ell-1}$. In this case we write $\text{deg}(v)=\ell$. The degree of a vector field is defined pointwise and can take different values at different points.

Let $(N,(\mh^1,\ldots, \mh^s))$ be an equiregular graded manifold. Take $p\in N$ and consider an open neighborhood $U$ of $p$  where a local frame $\{X_{1},\cdots,X_{n_1}\}$  generating  $\mathcal{H}^1$ is defined. Clearly the degree of $X_j$, for $j=1,\ldots,n_1$, is equal to one since the vector fields $X_1,\ldots,X_{n_1}$ belong to $\mh^1$. Moreover the vector fields $X_1, \ldots,X_{n_1}$ also lie in $\mh^2$, we add some vector fields $X_{n_{1}+1},\cdots,X_{n_2} \in \mathcal{H}^2\setminus \mathcal{H}^{1} $ so that $(X_1)_p,\ldots,(X_{n_2})_p$ generate $\mathcal{H}^2_p$. Reducing $U$ if necessary we have that $X_1,\ldots,X_{n_2}$ generate $\mathcal{H}^2$ in $U$.  Iterating this procedure we obtain a basis of $TM$ in a neighborhood of $p$
\begin{equation}
\label{local adapted basis to the bundle}
 (X_1,\ldots,X_{n_1},X_{n_1+1},\ldots,X_{n_2},\ldots,X_{n_{s-1}+1}, \ldots,X_n),
\end{equation}
such that the vector fields $X_{n_{i-1}+1},\ldots,X_{n_i}$ have degree equal to $i$, where $n_0:=0$. The basis obtained in (\ref{local adapted basis to the bundle}) is called an \textit{adapted basis} to the filtration $(\mh^1,\ldots,\mh^s)$. 

Given an adapted basis $(X_i)_{1\le i\le n}$, the \emph{degree} of the \emph{simple} $m$-vector field $X_{j_1}\wedge\ldots\wedge X_{j_m}$ is defined by
\[
\deg(X_{j_1}\wedge\ldots\wedge X_{j_m}):=\sum_{i=1}^m\deg(X_{j_i}).
\]
Any $m$-vector $X$ can be expressed as a sum
\[
X_p=\sum_{J}\lambda_J(p)(X_J)_p,
\]
where $J=(j_1,\ldots,j_m)$, $1\le j_1<\cdots <j_m\le n$, is an ordered multi-index, and $X_J:=X_{j_1}\wedge\ldots\wedge X_{j_m}$. The degree of $X$ at $p$ with respect to the adapted basis $(X_i)_{1\le i\le n}$ is defined by
\[
\max\{\deg((X_J)_p):\la_J(p)\neq 0\}.
\]
It can be easily checked that the degree of $X$ is independent of the choice of the adapted basis and it is denoted by $\deg(X)$.

If $X=\sum_J\la_J X_J$ is an $m$-vector expressed as a linear combination of simple $m$-vectors $X_J$, its projection onto the subset of $m$-vectors of degree $d$ is given by
\begin{equation}
\label{projection onto degree d}
(X)_d=\sum_{\deg(X_J)=d} \la_JX_J,
\end{equation}
and its projection over the subset of $m$-vectors of degree larger than $d$ by
\[
\pi_d(X)=\sum_{\deg(X_J)\ge d+1} \la_JX_J.
\]

In an equiregular graded manifold with a local adapted basis $(X_1, \ldots,X_n)$, defined as in (\ref{local adapted basis to the bundle}), the maximal degree that can be achieved by an $m$-vector, $m\le n$,  is the integer $d_{\max}^m$ defined by 
\begin{equation}
\label{maximum degree over N}
 d_{\max}^m:=\deg(X_{n-m+1})+\cdots+\deg(X_{n}).
\end{equation}

\subsection{Degree of a submanifold}
Let $M$ be a submanifold of class $C^1$ immersed in an equiregular graded manifold $(N,(\mh^1,\ldots, \mh^s))$ such that $\dim(M)=m<n=\dim(N)$. Then, following \cite{DonneMagnani, vittonemagnani}, we define the degree of $M$ at a point $p\in M$ by 
\[
\deg_M(p):=\deg(v_1\wedge\ldots\wedge v_m),
\]
where $v_1,\ldots,v_m$ is a basis of $T_pM$. Obviously, the degree is independent of the choice of the basis of $T_pM$. Indeed, if we consider another basis $\mathcal{B'}=(v_1', \cdots, v_m')$ of  $T_p M$, we get 
$$v_1 \wedge \cdots \wedge v_m= \det(M_{\mathcal{B},\mathcal{B'}}) \ v_1' \wedge \cdots \wedge v_m',$$
where $M_{\mathcal{B},\mathcal{B'}}$ denotes the change of basis matrix. Since $\det(M_{\mathcal{B},\mathcal{B'}})\ne 0$, we conclude that $\deg_M(p)$ is well-defined.
The \textit{degree} $\deg(M)$ of a submanifold $M$ is the integer 
\[
\deg(M):=\max_{p\in M} \deg_{M}(p).
\]
We define the \textit{singular set} of a submanifold $M$ by 
\begin{equation}
\label{singular set}
 M_0=\{p \in M : \deg_M(p)<\deg(M) \}.
\end{equation}
Singular points can have different degrees between $m$ and $\deg(M)-1$.

In \cite[0.6.B]{Gromov} Gromov considers the flag 
\begin{equation}
\label{tangent flag of M}
 \tilde{\mathcal{H}}_p^1 \subset \tilde{\mathcal{H}}_p^2 \subset\cdots\subset \tilde{\mathcal{H}}_p^s =T_pM,
\end{equation}
where $\tilde{\mathcal{H}}_p^j=T_pM \cap \mathcal{H}_p^j $ and 
$\tilde{m}_j=\text{dim}(\tilde{\mathcal{H}}_p^j)$. Then he defines the degree at $p$ by
\begin{equation*}
 \tilde{D}_H(p)= \sum_{j=1}^s j (\tilde{m}_j- \tilde{m}_{j-1}),
\end{equation*}
setting $\tilde{m}_{0}=0$.
It is easy to check that our definition of degree is equivalent to Gromov's one, see \cite[Chapter 2.2]{thesis}. As we already pointed out, $(M,(\tilde{\mh}^j)_{j \in \nn})$ is a graded manifold.

Let us check now that the degree of a vector field and the degree of points in a submanifold are lower semicontinuous functions.

\begin{lemma}
\label{vectors semicont degree}
Let $(N,(\mh^1,\ldots, \mh^s))$ be a graded manifold regular at $p\in N$. Let $V$ be a vector field defined on a open neighborhood $U_1$ of $p$. Then we have
 \[
 \liminf\limits_{q\rightarrow p}\deg(V_q)\geqslant \deg(V_p).
\]
\end{lemma}
\begin{proof} 
As $p\in N$ is regular, there exists a local adapted basis $(X_1,\ldots,X_n)$ in an open neighborhood $U_2\subset U_1$ of $p$. We express the smooth vector field $V$ in $U_2$ as
 \begin{equation}
  V_q=\sum_{i=1}^{s} \sum_{j=n_{i-1}+1}^{n_{i}} c_{ij}(q) (X_j)_q
  \label{vector written in basis}
\end{equation}
on $U_2$  with respect to an  adapted basis $(X_1,\cdots, X_n)$, where $c_{ij}\in C^\infty(U_2)$.
Suppose that the degree $\deg(V_p)$ of $V$ at $p$ is equal to $d \in \nn$. Then, 
there exists an integer $k \in \{ n_{d-1}+1,\cdots,n_{d}\}$  such that
$c_{dk}(p)\ne 0$ and $c_{ij}(p)=0$ for all $i=d+1,\cdots,s$ and $j=n_{i-1}+1,\cdots,n_{i}$. 
By continuity, there exists an open neighborhood $U'\subset U_2$ such that $c_{dk}(q)\ne 0$
for each $q$ in $U'$. Therefore for each $q$ in $U'$ the degree of $V_q$ is greater than or equal to the degree of $V(p)$,
\[
 \deg(V_q)\geqslant \deg(V_p)=d.
\]
Taking limits we get
\[
\liminf\limits_{q\rightarrow p}\deg(V_q)\geqslant \deg(V_p). \qedhere
\]
\end{proof}
\begin{remark}
In the proof of Lemma~\ref{vectors semicont degree}, $\deg(V_q)$ could be strictly greater than $d$ in case there were a coefficient $c_{ij}$ with $i\ge d+1$ satisfying $c_{ij}(q)\neq 0$. 
\end{remark}

\begin{proposition}
\label{semicont of the degree}
 Let $M$ be a $C^1$ immersed submanifold  in a graded manifold $(N,(\mh^1,\ldots, \mh^s))$. Assume that $N$ is regular at $p\in M$. Then we have 
 \[
  \liminf\limits_{q\rightarrow p,  q \in M}\deg_M(q)\geqslant \deg_M(p).
 \]
\end{proposition}

\begin{proof}
The proof imitates the one of Lemma \ref{vectors semicont degree} and it is
 based on the fact that the degree is defined by an open condition. Let
$\tau_M=\sum_J\tau_JX_J$ be a tangent $m$-vector in an open neighborhood $U$ of $p$, where a local adapted basis is defined. The functions $\tau_J$ are continuous on $U$. Suppose that the degree $\deg_M(p)$ at $p$ in $M$ is equal to $d$. This means that there exists a multi-index $\bar{J}$ such that $\tau_{\bar{J}}(p)\ne0$ and $\deg((X_{\bar{J}})_p)=d$. Since the function $\tau_{\bar{J}}$ is continuous there exists a neighborhood $U'\subset U$ such that $\tau_{\bar{J}}(q)\ne0$ in $U'$. Therefore, $\deg(\tau_M(q))\ge d$ and taking limits we have 
\[
\liminf\limits_{q\rightarrow p}\deg_M(q)\geqslant \deg_M(p). \qedhere
\]
\end{proof}

\begin{corollary}
\label{lsc cor}
Let $M$ be a $C^1$ submanifold immersed in an equiregular graded manifold. Then 
\begin{enumerate}
\item $\deg_M$ is a lower semicontinuous function on $M$.
\item The singular set $M_0$ defined in \eqref{singular set} is closed in $M$.
\end{enumerate} 
\end{corollary}

\begin{proof}
The first assertion follows from Proposition~\ref{semicont of the degree} since every  point in an equiregular graded  manifold is regular. To prove 2, we take $p\in M\smallsetminus M_0$. By 1, there exists a open neighborhood $U$ of $p$ in $M$ such that each point $q$ in $U$ has degree $\deg_M(q)$ equal to $\deg(M)$. Therefore we have $U\subset M\smallsetminus M_0$ and hence $M\smallsetminus M_0$ is an open set.  
\end{proof}
\subsection{Carnot manifolds}
Let $N$ be an $n$-dimensional smooth manifold. An \emph{$l$-dimensional distribution} $\mathcal{H}$ on $N$ assigns smoothly to every $p\in N$ an $l$-dimensional vector subspace $\mathcal{H}_p$ of $T_pN$. We say that a distribution $\mathcal{H}$ complies \emph{H\"ormander's condition} if  any local frame $\{X_1, \ldots, X_l\}$ spanning $\mathcal{H}$ satisfies
\[
\dim(\mathcal{L}(X_1,\ldots,X_{l}))(p)=n, \quad \text{for all} \ p\in N,
\]
where $\mathcal{L}(X_{1},\ldots,X_{l})$ is the linear span of the vector fields  $X_{1},\ldots,X_{l}$ and their  commutators of any order.

A \emph{Carnot manifold} $(N,\mathcal{H})$ is a smooth manifold $N$ endowed with an $l$-dimensional distribution $\mathcal{H}$ satisfying H\"ormander's condition. We refer to $\mathcal{H}$ as the \emph{horizontal distribution}. We say that a vector field on $N$ is \emph{horizontal} if it is tangent to the horizontal distribution at every point. A $C^1$ path is horizontal if the tangent vector is everywhere tangent to the horizontal distribution. A \emph{sub-Riemannian manifold} $(N,\mathcal{H},h)$ is a Carnot manifold $(N,\mathcal{H})$ endowed with a positive-definite inner product $h$ on $\mathcal{H}$. Such an inner product can always be extended to a Riemannian metric on $N$. Alternatively, any Riemannian metric on $N$ restricted to $\mathcal{H}$ provides a structure of sub-Riemannian manifold. Chow's Theorem assures that in a Carnot manifold $(N,\mathcal{H})$ the set of points that can be connected to a given point $p\in N$ by a horizontal path is the connected component of $N$ containing $p$, see \cite{Montgomery}.
Given a Carnot manifold $(N,\mathcal{H})$, we have a flag of subbundles
\begin{equation}
\mathcal{H}^1:=\mathcal{H}\subset\mathcal{H}^2\subset\cdots\subset \mathcal{H}^i\subset\cdots\subset TN,
\label{eq:c manifold flag}
\end{equation}
defined by 
\begin{equation*}
\mathcal{H}^{i+1} :=\mathcal{H}^i + [\mathcal{H},\mathcal{H}^i], \qquad i\ge 1,
\end{equation*}
where
\begin{equation*}
 [\mathcal{H},\mathcal{H}^i]:=\{[X,Y]  : X \in \mathcal{H},Y \in \mathcal{H}^i\}.
\end{equation*}
The smallest integer $s$ satisfying $\mathcal{H}^s_p=T_pN$ is called the \textit{step} of the distribution $\mathcal{H}$ at the point $p$. Therefore, we have 
\begin{equation*}
 \mathcal{H}_p\subset\mathcal{H}^2_p\subset\cdots\subset \mathcal{H}^s_p=T_p N.
\end{equation*}
 The integer list $(n_1(p),\cdots,n_s(p))$  is called the\textit{ growth vector} of $\mathcal{H}$ at $p$. When the growth vector is constant in a neighborhood of a point $p \in N$ we say that $p$ is a \textit{regular point} for the distribution. This flag of sub-bundles \eqref{eq:c manifold flag} associated to a Carnot manifold $(N,\mh)$ gives rise to the graded structure $(N,(\mh^i))$. Clearly an equiregular Carnot manifold $(N,\mh)$ of step $s$ is an equiregular graded manifold $(N,\mh^1, \ldots, \mh^s)$. In particular a Carnot group turns out to be an equiregular graded manifold.

Given a connected sub-Riemannian manifold $(N,\mathcal{H},h)$, and a $C^1$ horizontal path $\gamma:[a,b]\to N$, we define the length of $\gamma$ by
\begin{equation}
L(\gamma)=\int_a^b \ \sqrt{h(\dot{\gamma}(t),\dot{\gamma}(t))} \ dt.
 \label{length fun}
\end{equation}
By means of the equality
\begin{equation}
 d_c(p,q):=\inf \{L(\gamma) :  \gamma\  \text{is a } C^1\ \text{horizontal path joining } p,q  \in N \},
 \label{C-C distance}
\end{equation}
this length defines a distance function (see \cite[\S~2.1.1,\S~2.1.2]{BuragoYuri}) usually called the \textit{Carnot-Carathéodory distance}, or CC-\emph{distance} for short. See \cite[Chapter 1.4]{Montgomery} for further details. 

\section{Area for submanifolds of given degree}
\label{sec:area}

In this section we shall consider a graded manifold $(N,\mh^1,\ldots,\mh^s)$ endowed with a Riemannian metric $g$, and an immersed submanifold $M$ of dimension $m$.

We recall the following construction from \cite[1.4.D]{Gromov}: given $p\in N$, we recursively define the subspaces $\mathcal{K}^1_p:=\mathcal{H}_p$, $\mathcal{K}^{i+1}_p:=(\mathcal{H}_p^i)^\perp\cap \mathcal{H}^{i+1}_p$, for $1\le i\le (s-1)$. Here $\perp$ means perpendicular with respect to the Riemannian metric $g$. Therefore we have the decomposition of $T_pN$ into orthogonal subspaces
\begin{equation}
\label{eq:decomp}
T_pN=\mathcal{K}_p^1\oplus \mathcal{K}_p^2\oplus\cdots \oplus\mathcal{K}_p^s.
\end{equation}
Given $r>0$, a unique Riemannian metric $g_r$ is defined under the conditions: (i) the subspaces $\mathcal{K}_i$ are orthogonal, and (ii)
\begin{equation}
\label{metric blow-up} 
g_r|_{\mathcal{K}_i}=\frac{1}{r^{i-1}}g|_{\mathcal{K}_i}, \qquad i=1,\ldots,s.
\end{equation}
When we consider Carnot manifolds, it is well-known that the Riemannian distances of $(N,g_r)$ uniformly converge to the Carnot-Carathéodory distance of $(N,\mathcal{H},h)$, \cite[p.~144]{Gromov}.

Working on a neighborhood $U$ of $p$ where a local frame $(X_1,\ldots,X_k)$  generating the distribution $\mh$ is defined, we construct  an \emph{orthonormal} adapted basis $(X_1,\ldots,X_n)$ for the Riemannian metric $g$ by choosing orthonormal bases in the orthogonal subspaces $\mathcal{K}^i$, $1\le i\le s$.
Thus, the $m$-vector fields
\begin{equation}
\label{eq:Xtilde}
 \tilde{X}_J^r=\left(r^{\frac{1}{2}(\deg(X_{j_1})-1)}X_{j_1}\right)
 \wedge \ldots \wedge \left(r^{\frac{1}{2}(\deg(X_{j_m})-1)}X_{j_m}\right),
\end{equation}
where $J = (j_1 , j_2 , \ldots , j_m )$ for $1 \leqslant j_1 < \cdots < j_m \leqslant n$, are orthonormal with respect to the extension of the metric $g_r$ to the space of $m$-vectors. We recall that the metric $g_{r}$ is extended to the space of $m$-vectors simply defining
\begin{equation}
g_r(v_1\wedge \ldots \wedge v_m , v_1'\wedge \ldots \wedge v_m')=\det\big(g_r(v_i,v_j')\big)_{1\le i,j\le m},
\label{Leibniz}
\end{equation}
for $v_1,\ldots,v_m$ and $v_1', \ldots,v_m'$ in $T_{p} N$. Observe that the extension is denoted the same way.

\subsection{Area for submanifolds of given degree}
Assume now that $M$ is an immersed submanifold of dimension $m$ in a equiregular graded manifold $(N,\mathcal{H}^1,\ldots, \mh^s)$ equipped with the Riemannian metric $g$. We take a Riemannian metric $\mu$ on $M$. For any $p\in M$ we pick a $\mu$-orthonormal basis $e_1,\ldots,e_m$ in $T_pM$. By the area formula we get
\begin{equation}
\label{eq:sraintegral}
A(M',g_r)
=\int_{{M'}} \left| e_1 \wedge \ldots \wedge e_m\right|_{g_r} d\mu(p),
\end{equation}
where $M'$ is a bounded measurable subset of $M$ and $A(M',g_r)$ is the $m$-dimensional area of $M'$ with respect to the Riemannian metric $g_r$.

Now we express
\[
e_1\wedge \ldots \wedge e_m=\sum_{J} \tau_{J}(p) {(X_{J})_p}=\sum_{J} \tilde{\tau}^r_{J}(p) (\tilde{X}_{J}^r)_p,\quad r>0.
\]
From \eqref{eq:Xtilde} we get $\tilde{X}_{J}^r=r^{\frac{1}{2}(\deg(X_J)-m)}{X_{J}}$, and so  $\tilde{\tau}_{J}=r^{-\frac{1}{2}(\deg(X_J)-m)}{\tau_{J}}$.
Moreover, as $\{\tilde{X}_{J}^r\}$ is an orthonormal basis for $g_r$, we have 
\[
\left| e_1 \wedge \ldots \wedge e_m\right|_{g_r}^2=\sum_J (\tilde{\tau}_{J}^r(p))^2=
\sum_J r^{-(\deg(X_J)-m)}{\tau_{J}^2(p)}.
\]
Therefore, we have 
\begin{align*}
 \lim_{r\downarrow 0} r^{\frac{1}{2}(\deg(M)-m)} \left| e_1 \wedge \ldots \wedge e_m\right|_{g_r}&=
 \lim_{r\downarrow0}  \Big(\sum_J r^{(\deg(M)-\deg(X_J))}{\tau_{J}^2(p)}\Big)^{1/2}
 \\
 &=\Big(\sum_{\deg(X_J)=\deg(M)} \tau_{J}^2(p)\Big)^{1/2}.
\end{align*}

By Lebesgue's dominated convergence theorem we obtain 
\begin{equation}
\label{eq:integral_formula_Ad}
\lim_{r\downarrow 0} \Big(r^{\tfrac{1}{2}(\deg(M)-m)}A(M',g_r)\Big)=\int_{{M'}} 
\Big(\sum_{\deg(X_J)=\deg(M)} \tau_{J}^2(p)\Big)^{\frac{1}{2}} \ d\mu(p).
\end{equation}

\begin{definition}
If $M$ is an immersed submanifold of degree $d$ in an  equiregular graded manifold $(N,\mathcal{H}^1,\ldots, \mh^s)$ endowed with a Riemannian metric $g$, the degree $d$ area $A_d$ is defined by
\[
A_d(M'):=\lim_{r\downarrow 0} \Big(r^{\tfrac{1}{2}(d-m)}A(M',g_r)\Big),
\]
for any bounded measurable set $M'\subset M$.
\end{definition}

Equation \eqref{eq:integral_formula_Ad} provides an integral formula for the area $A_d$. An immediate consequence of the definition is the following

\begin{remark}
 Setting $d:=\deg(M)$ we have by equation \eqref{eq:integral_formula_Ad} and the notation introduced in \eqref{projection onto degree d}  that the degree $d$ area $A_d$ is given by 
 \begin{equation}
 \label{eq:projected_integral_formula_Ad}
  A_d(M')=\int_{{M'}} 
 |\left( e_1\wedge \ldots \wedge e_m \right)_d|_{g} \ d\mu(p).
 \end{equation}
 for any bounded measurable set $M'\subset M$. When the ambient manifold is a Carnot group this area formula was obtained by \cite{vittonemagnani}. Notice that the $d$ area $A_d$ is given by the integral of the $m$-form
 \begin{equation}
  \label{eq:dmform}
  \omega_d(v_1,\ldots,v_m)(p)=\escpr{v_1\wedge \ldots \wedge v_m, \dfrac{(e_1\wedge \ldots \wedge e_m )_d}{|(e_1\wedge \ldots \wedge e_m)_d|_g} },
 \end{equation} 
 where $v_1,\ldots,v_m$ is a basis of $T_p M$. 

In a more general setting, an $m$-dimensional submanifold in a Riemannian manifold is an $m$-current (i.e., an element of the dual of the space of $m$-forms), and the area is the mass of this current (for more details see \cite{GMTFe}). Similarly, a natural generalization of an $m$-dimensional submanifold of degree $d$ immersed in a graded manifold  is an $m$-current of degree $d$ whose mass should be given by $A_d$. In \cite{FSS07} the authors studied the theory of $\mathbb{H}$-currents in the Heisenberg group. Their mass coincides with our area \eqref{eq:projected_integral_formula_Ad} on intrinsic $C^1$ submanifolds. However in \eqref{eq:dmform} we consider all possible $m$-forms and not only the intrinsic $m$-forms in the Rumin's complex \cite{Rumin94,Pansu93, BaldiFranchi12}.
\end{remark}

\begin{corollary}
Let $M$ be an $m$-dimensional immersed submanifold of degree $d$ in a graded manifold $(N,\mh^1,\ldots,\mh^s)$ endowed with a Riemannian metric $g$. Let $M_0\subset M$ be the closed set of singular points of $M$. Then $A_d(M_0)=0$.
\end{corollary}

\begin{proof}
Take an orthonormal basis $v_1,\ldots,v_m$ of $M$ at $p$ and express $v_1\wedge\ldots\wedge v_m=\sum_{J} \tau_J(p)(X_J)_p$. When $p$ is a singular point, $\deg(v_1\wedge\ldots\wedge v_m)<\deg(M)=d$ and so $\tau_J(p)=0$ whenever $\deg(X_J)\ge d$.

Since $M_0$ is measurable, from \eqref{eq:integral_formula_Ad} we obtain
\[
A_d(M_0)=\int_{M_0}
\Big(\sum_{\deg(X_J)=d} \tau_{J}^2(p)\Big)^{\frac{1}{2}} \ d\mu(p)
\]
and so $A_d(M_0)=0$.
\end{proof}

\begin{remark}
Another easy consequence of the definition is the following: if $M$ is an immersed submanifold of degree $d$ in graded manifold $(N,\mh^1,\ldots,\mh^s)$ with a Riemannian metric, then $A_{d'}(M')=\infty$ for any open set $M'\subset M$ when $d'<d$. This follows easily since in the expression
\[
r^{\frac{1}{2} (d'-m)} \left| e_1 \wedge \ldots \wedge e_m\right|_{g_r}
\]
we would have summands with negative exponent for $r$.
\end{remark}

In the following example, we exhibit a Carnot manifold with two different Riemannian metrics that coincide when restricted to the horizontal distribution, but yield different area functionals of a given degree

\begin{example}
\label{ex:dependence on the metric extension}
We consider the Carnot group $\mathbb{H}^1 \otimes \mathbb{H}^1$, which is the direct product of two Heisenberg groups. Namely, let  $\rr^3 \times \rr^3$ be the $6$-dimensional Euclidean space with  coordinates $(x,y,z,x',y',z')$. We consider the $4$-dimensional distribution $\mh$ generated by
\begin{align*}
X&=\partial_x-\dfrac{y}{2}\partial_z,  &   Y&=\partial_y+\dfrac{x}{2}\partial_z,\\
X'&=\partial_{x'}-\dfrac{y'}{2}\partial_{z'} &Y'&=\partial_{y'}+\dfrac{x'}{2} \partial_{z'}.
\end{align*}
The vector fields $Z=[X,Y]=\partial_z$ and $Z'=[X',Y']=\partial_{z'}$ are the only non trivial commutators that generate, together with $X,Y,X',Y'$, the subspace $\mathcal{H}^2=T(\hh^1\otimes\hh^1)$. Let $\Omega$ be a bounded open set of $\rr^2$ and $u$  a smooth function on $\Omega$ such that $u_t(s,t)\equiv0$. We consider the immersed surface 
\begin{align*}
\Phi:\Omega &\longrightarrow \mathbb{H}^1 \otimes \mathbb{H}^1, \\
(s,t)&\longmapsto (s,0,u(s,t),0,t,u(s,t)),
\end{align*}
whose tangent vectors  are 
\begin{align*}
\Phi_s=& (1,0,u_s,0,0,u_s)=X+u_s \ Z+u_s \ Z',\\
\Phi_t=& (0,0,0,0,1,0)=Y'.
\end{align*}
Thus, the $2$-vector tangent to $M$ is given by  
\[
\Phi_s\wedge\Phi_t=X\wedge Y'+ u_s(Z\wedge Y'+Z'\wedge Y').
\]

When $u_s(s,t)$ is different from zero the degree is equal to $3$, since both $Z\wedge Y'$ and $Z'\wedge Y'$ have degree equal to $3$. Points of degree $2$ corresponds to the zeroes of $u_s$. We define a $2$-parameter family $g_{\lambda,\nu}$ of Riemannian metrics on $\hh^1\otimes\hh^1$, for $(\la,\mu)\in\rr^2$, by the conditions (i) $(X,Y,X',Y')$ is an orthonormal basis of $\mh$, (ii) $Z$,  $Z'$ are orthogonal to $\mh$, and (iii) $g(Z,Z)= \lambda$, $g(Z',Z')= \mu$ and $g(Z',Z)=0$. Therefore, the degree $3$ area of $\Omega$ with respect to the metric $g_{\mu,\nu}$ is given by 
\[
A_3(\Omega)=\int_{\Om} u_s( \lambda+ \nu) \ ds dt.
\]
As we shall see later, these  different functionals will  not have the same critical points, that would depend on the election of Riemannian metric.  
\end{example}

\section{Examples}
\label{sc:examples}

\subsection{Degree of a hypersurface in a Carnot manifold} \label{ssc:hyp}
Let $M$ be a $C^1$ hypersurface immersed in an equiregular Carnot manifold $(N,\mh)$, where $\mh$ is a bracket generating $l$-dimensional distribution. Let $Q$ be the homogeneous dimension of $N$ and $p\in M$. 

Let us check that $\deg(M)=Q-1$. The pointwise degree of $M$ is given by 
\[
\deg_{M}(p)=\sum_{j=1}^s j (\tilde{m}_j-\tilde{m}_{j-1}) ,
\]
where $\tilde{m}_j=\text{dim}(\tilde{\mh}_p^j)$ with $\tilde{\mh}_p^j=T_p M \cap \mathcal{H}_p^j$. Recall that $n_i=\dim(\mh_p^i)$. As $T_pM$ is a hyperplane of $T_pN$ we have that either $\tilde{\mh}_p^i=\mh^i_p$ and $\tilde{m}_i=n_i$, or $\tilde{\mh}_p^i$ is a hyperplane of $\mh_p^i$ and $\tilde{m}_i=m_i-1$. On the other hand,
\[
\tilde{m}_i-\tilde{m}_{i-1}\le n_i-n_{i-1}.
\]
Writing 
\[
 n_i-n_{i-1}=\tilde{m}_i-\tilde{m}_{i-1}+z_i,
\]
for non-negative integers $z_i$ and adding up on $i$ from $1$ to $s$ we get
\[
\sum_{i=1}^s z_i=1,
\]
since $\tilde{m}_s=n-1$ and $n_s=n$. We conclude that there exists $i_0\in\{1,\ldots,s\}$ such that $z_{i_0}=1$ and $z_j=0$ for all $j\neq i_0$. This implies
\begin{alignat*}{2}
\tilde{m}_i&=n_i , && \qquad i<i_0,
\\
\tilde{m}_i&=n_i -1, &&\qquad i\ge i_0.
\end{alignat*}

If $i_0>1$ for all $p\in M$, then $\mh\subset TM$, a contradiction since $\mh$ is a bracket-generating distribution. We conclude that $i_0=1$ and so 
\begin{align*}
\deg(M)&=\sum_{i=1}^s i\,(\tilde{m}_i-\tilde{m}_{i-1})=1\cdot\tilde{m}_1+\sum_{i=2}^s i\,(\tilde{m}_i-\tilde{m}_{i-1})
\\
&=1\cdot (n_1-1)+\sum_{i=2}^s i\,(n_i-n_{i-1})=Q-1.
\end{align*}

\subsection{$A_{2n+1}$-area of a hypersurface in a $(2n+1)$-dimensional contact manifold}
\label{contact geometry}
A \textit{contact manifold} is a smooth manifold $M^{2n+1}$ of odd dimension endowed with a one form $\omega$ such that $d\omega$ is non-degenerate when  restricted to $\mathcal{H}=\text{ker}(\omega)$. Since it holds 
\[
 d\omega(X,Y)=X(\omega(Y))-Y(\omega(X))-\omega([X,Y]),
\]
for $X,Y\in\mathcal{H}$, the distribution $\mathcal{H}$ is non-integrable and satisfies H\"{o}rmander rank condition by Frobenius theorem. When we define a horizontal metric $h$ on the distribution $\mathcal{H}$ then $(M,\mathcal{H},h)$ is a sub-Riemannian structure. It is easy to prove that there exists an unique vector field $T$ on $M$ so that 
\[
 \omega(T)=1, \quad \mathcal{L}_T(X)=0,
\]
where $\mathcal{L}$ is the Lie derivative and $X$ is any vector field on $M$. This vector field $T$ is called the \textit{Reeb vector field}. We can always extend the horizontal metric $h$ to the Riemannian metric $g$ making $T$  a unit vector orthogonal to $\mathcal{H}$.

Let $\Sigma$ be a $C^1$ hypersurface immersed in $M$. In this setting the singular set of $\Sigma$ is given by 
\[
 \Sigma_0=\{p \in \Sigma : T_p \Sigma=\mathcal{H}_p \},
\]
and corresponds to the points in $\Sigma$ of degree $2n$. Observe  that the non-integrability of $\mathcal{H}$ implies that the set $\Sigma\smallsetminus \Sigma_0$ is not empty in any hypersurface $\Sigma$.

Let $N$ be the unit vector field normal to $\Sigma$ at each point, then on the regular set $\Sigma \smallsetminus \Sigma_0$ the $g$-orthogonal projection $N_h$ of $N$ onto the distribution $\mathcal{H}$ is different from zero. Therefore out of the singular set 
$\Sigma_0$ we define the \textit{horizontal unit normal} by 
\[
 \nu_h=\dfrac{N_h}{|N_h|},
\]
and the vector field 
\[
 S=\langle N,T \rangle \nu_h-|N_h|T,
\]
which is tangent to $\Sigma$ and belongs to $\mathcal{H}^2$. Moreover, $T_p\Sigma \cap (\mathcal{H}^2_p\smallsetminus \mathcal{H}^1_p)$ has
dimension equal to one and $T_p\Sigma \cap \mathcal{H}_p^1$ equal to $2n-1$, thus the degree of the hypersurface $\Sigma$ out of the singular set is equal to $2n+1$. Let $e_1,\ldots,e_{2n-1}$ be an orthonormal basis in $T_p\Sigma\cap\mathcal{H}^1_p$. Then $e_1,\ldots,e_{2n-1},S_p$ is an orthonomal basis of $T_p\Sigma$ and we have 
\[
 e_1\wedge \ldots\wedge e_{2n-1}\wedge S=\escpr{N,T}e_1\wedge \ldots\wedge e_{2n-1}\wedge \nu_h- |N_h|e_1\wedge \ldots\wedge e_{2n-1}\wedge T.
\]
Hence we obtain 
\begin{equation}
A_{2n+1}(\Sigma)=\int_{\Sigma} |N_h| d\Sigma.
\end{equation}
In \cite{Galli} Galli obtained this formula as the perimeter of a set that has $C^1$ boundary $\Sigma$ and in \cite{Nataliya} Shcherbakova as the limit of the volume of a $\eps$-cylinder around $\Sigma$ over its height equal to $\eps$. This formula was obtain for surfaces in a $3$-dimensional pseudo-hermitian manifold in \cite{ChengMalchiodi} and by S. Pauls in \cite{scottpauls}. This is exactly the area formula independently established in recent years in the Heisenberg group $\hh^n$, that is the prototype for contact manifolds (see for instance \cite{DanielliGarofalo,ChengMalchiodi,ChengFangYang,RitoreRosales,HladPauls}).

\begin{example}[The roto-translational group]
Take coordinates $(x,y,\theta)$ in the $3$-dimensional manifold $\rr^2\times\mathbb{S}^1$. We consider the contact form
\[
 \omega=\si dx- \co dy,
\]
the horizontal distribution $\mathcal{H}=\text{ker}(\omega)$,  is spanned by the vector fields 
\[
X=\cos(\theta) \partial_x+ \sin(\theta) \partial_y, \quad  Y=\partial_{\theta},
\]
and the horizontal metric $h$ that makes $X$ and $Y$ orthonormal.

Therefore $\rr^2\times \mathbb{S}^1$ endowed with this one form $\omega$
is a contact manifold. Moreover $(\rr^2\times \mathbb{S}^1, \mathcal{H},h)$ has a sub-Riemannian structure which is also a Lie group known as the roto-translational group. A mathematical model of simple cells of the visual cortex V1 using the sub-Riemannian geometry 
of the roto-translational Lie group was proposed by Citti and Sarti (see \cite{cittilibro}, \cite{citti}). Here the Reeb vector field is given by
\[
 T=[X,Y]=\sin(\theta) \partial_x-\cos(\theta) \partial_y.
\]
Let $\Omega$ be an open set of $\rr^{2}$ and $u:\Omega\rightarrow \rr$ be a function of class $C^1$.
When we consider a graph $\Sigma=\text{Graph}(u)$ given by the zero set level of the $C^1$ function 
$$f(x,y,\theta)=u(x,y)-\theta=0,$$
the projection of the unit normal $N$ onto the horizontal distribution is given by 
\[
 N_h=\dfrac{ X(u)X-Y}{\sqrt{1+X(u)^2+T(u)^2}}.
\]
Hence the $3$-area functional is given by
\begin{equation*}
A_3(\Sigma , \lambda)=\int_{\Omega}\left(1+X(u)^2\right)^{\frac{1}{2}} \, dx dy.
\end{equation*}
\end{example}
\subsection{$A_4$-area of a ruled surface  immersed in an Engel structure}
\label{sc:2jetspace}
Let $E=\rr^2  \times  \mathbb{S}^1 \times \rr$ be a smooth manifold with coordinates $p=(x,y,\theta,k)$. We set $\mathcal{H}=\text{span}\{X_1,X_2\}$, where 
\begin{equation}
 X_{1}= \cos( \theta ) \partial_{x} + \sin( \theta )  \partial_{y}+ k\partial_{\theta}, \qquad X_{2}=\partial_{k}.
 \label{vector fields 1layer}
\end{equation}
Therefore $(E,\mh)$ is a Carnot manifold, indeed  $\mh$ satisfy the H\"{o}rmander rank condition since $X_1$ and $X_2$
\begin{equation}
\begin{aligned}
 X_{3}&=[X_{1},X_{2}]=-\partial_{\theta}\\
 X_{4}&=[X_{1},[X_{1},X_{2}]]=-\sin(\theta)\partial_{x}+ \cos(\theta) \partial_{y}
 \label{vector fields 2layer}
\end{aligned}
\end{equation}
generate all the tangent bundle.
Here we follow a computation developed by Le Donne and Magnani in \cite{DonneMagnani} in the Engel group. Let $\Omega$ be an open set of $\rr^2$ endowed with the Lebesgue measure. Since we are particularly interested in applications to the visual cortex (see \cite{thesis},\cite[1.5.1.4]{Petitot2014} to understand the reasons)
we consider the immersion $\Phi:\Omega \rightarrow E$ given by $\Phi=(x,y,\theta(x,y),\kk(x,y))$ and we set $\Sigma=\Phi(\Omega)$. The tangent vectors to 
$\Sigma$ are 
\begin{equation}
  \Phi_{x}=(1,0, \theta_{x}, k_{x}), \qquad \Phi_{y}=(0,1,\theta_{y},\kk_{y}).
  \label{vector fields tangent to Sigma}
\end{equation}

In order to know the dimension of $T_p \Sigma \cap \mathcal{H}_p$  it is necessary to take in 
account the rank of the matrix
\begin{equation}
 B=\left(\begin{array}{cccc}
  1 & 0 & \theta_{x}& \kk_{x}\\
  0 & 1 & \theta_{y}& \kk_{y}\\
  \cos(\theta) &\sin(\theta)&\kk&0\\
  0 & 0 & 0 & 1\\
 \end{array}\right).
\end{equation}
Obviously $\text{rank}(B)\geqslant3$, indeed we have 
$$\det\left(\begin{array}{ccc}
  1 & 0 & \kk_{x}\\
  0 & 1 & \kk_{y}\\
  0 & 0  & 1\\
 \end{array}\right) \ne 0.$$
Moreover, it holds
\begin{equation} 
 \begin{aligned}
 \text{rank}(B)=3 &\quad \Leftrightarrow \quad \det\left(\begin{array}{ccc}
  \cos(\theta) &\sin(\theta)&\kk\\
  1 & 0 & \theta_{x}\\
  0 & 1 & \theta_{y}\\
 \end{array}\right)=0\\
 &\quad \Leftrightarrow \quad \kk-\theta_{x}\cos(\theta)-\theta_{y}\sin(\theta)=0\\
 &\quad \Leftrightarrow \quad  \kk=X_1(\theta(x,y)).
 \end{aligned} 
 \label{foliation}
\end{equation}
Since we are inspired by the foliation property of hypersurface in the Heisenberg group and roto-translational
group, in the present work we consider only surface $\Sigma=\{(x,y,\theta(x,y),\kk(x,y))\}$ verifying the foliation
condition $\kk=X_1(\theta(x,y))$. Thus,
we have 
\begin{equation}
 \begin{aligned}
 \Phi_{x}\wedge\Phi_{y}=&(\cos(\theta)\kk_{y}-\sin(\theta)\kk_{x})X_{1}\wedge X_{2}-(\cos(\theta)\theta_{y} -\sin(\theta) \theta_{x})X_{1}\wedge X_{3}\\
                       &+X_{1}\wedge X_{4}+(\theta_{x}\kk_{y}-\theta_{y}\kk_{x}-\kk(\cos(\theta)\kk_{y}-\sin(\theta)\kk_{x}))X_{2}\wedge X_{3}\\
                       &+(\sin(\theta)\kk_{y}+\cos(\theta)\kk_{x})X_{2}\wedge X_{4}
\\&+(\kk-\sin(\theta)\theta_{y}-\cos(\theta)\theta_{x})X_{3}\wedge X_{4}.
\end{aligned}
\label{normalv2}
\end{equation}
By the foliation condition (\ref{foliation}) we have that the coefficient of $X_{3}\wedge X_{4}$ is always equal to zero, then we deduce that $\deg(\Sigma)\leqslant 4$.
Moreover, the coefficient of $X_{1}\wedge X_{4}$ never vanishes, therefore $\deg(\Sigma)=4$ and there are not singular points in $\Sigma$. 
When $\kk=X_1(\theta)$ a tangent basis of $T_p \Sigma$ adapted to \ref{tangent flag of M} is given by
\begin{equation}
\begin{aligned}
 e_1&=\co \Phi_x+ \si \Phi_y= X_1+X_1(\kk)X_2,\\
 e_2&=-\si \Phi_x+\co \Phi_y=X_4-X_4(\theta)X_3+X_4(\kk)X_2.
\end{aligned}
\label{adapted vector tangent to Sigma}
\end{equation}

When we fix the Riemannian metric $g_1$ that makes $(X_1,\ldots,X_4)$ orthonormal we have that the $A_4$-area of $\Sigma$ is given by 
\begin{equation}
\label{4 degree area in E1}
A_4(\Sigma,g)=\int_{\Omega}\left(1+X_1(\kk)^2\right)^{\frac{1}{2}} \ dx dy
 =\int_{\Omega}\left(1+X_1^2(\theta)^2\right)^{\frac{1}{2}} \ dx dy.
\end{equation}
When we fix the Euclidean metric $g_0$ that makes $(\partial_1,\partial_2, \partial_{\theta},\partial_k)$ we have that the $A_4$-area of $\Sigma$ is given by
\begin{equation}
\label{4 degree  area in E2}
A_4(\Sigma,g_0)=\int_{\Omega}\left(1+\kk^2+X_1(\kk)^2\right)^{\frac{1}{2}} \ dx dy.
\end{equation}

\section{Admissible variations for submanifolds}
\label{sc:admissible}
Let us consider an $m$-dimensional manifold $\bar{M}$ and an immersion $\Phi:\bar{M}\to N$ into an equiregular graded manifold endowed with a Riemannian metric $g=\escpr{\cdot,\cdot}$. We shall denote the image $\Phi(\bar{M})$ by $M$ and $d:=\deg(M)$. In this setting we have the following definition

\begin{definition}
\label{def:admissible}
A smooth map $\Gamma:\bar{M}\times (-\eps,\eps)\to N$ is said to be \emph{an admissible variation} of $\Phi$ if $\Gamma_t:\bar{M}\to N$, defined by $\Gamma_t(\bar{p}):=\Gamma(\bar{p},t)$, satisfies the following properties
\begin{enumerate}
\item[(i)] $\Gamma_0=\Phi$,
\item[(ii)] $\Gamma_t(\bar{M})$  is an immersion of the same degree as $\Phi(\bar{M})$ for small enough $t$, and
\item[(iii)] $\Gamma_t(\bar{p})=\Phi(\bar{p})$ for $\bar{p}$ outside a given compact subset of $\bar{M}$.
\end{enumerate}
\end{definition}

\begin{definition}
Given an admissible variation $\Gamma$, the \emph{associated variational vector field} is defined by
\begin{equation}
\label{eq:admissibleV}
V(\bar{p}):=\frac{\ptl\Gamma}{\ptl t}(\bar{p},0).
\end{equation}
\end{definition}

The vector field $V$ is an element of $\mathfrak{X}_0(\bar{M},N)$: i.e., a smooth map $V:\bar{M}\to TN$ such that $V(\bar{p})\in T_{\Phi(\bar{p})}N$ for all $\bar{p}\in\bar{M}$. It is equal to $0$ outside a compact subset of $\bar{M}$.

Let us see now that the variational vector field $V$ associated to an admissible variation $\Ga$ satisfies a differential equation of first order. Let $p=\Phi(\bar{p})$ for some $\bar{p}\in\bar{M}$, and $(X_1, \cdots, X_n)$ an adapted frame in a neighborhood $U$ of $p$. Take a basis $(\bar{e}_1,\ldots,\bar{e}_m)$ of $T_{\bar{p}}{\bar{M}}$ and let $e_j=d\Phi_{\bar{p}}(\bar{e}_j)$ for $1\le j\le m$. As $\Gamma_t(\bar{M})$ is a submanifold of the same degree as $\Phi(\bar{M})$ for small $t$, there follows
\begin{equation}
\label{flow that preserves degree sub}
 \left\langle (d  \Gamma_t)_{\bar{p}} (e_1) \wedge \ldots \wedge (d \Gamma_t)_{\bar{p}}(e_m) , {(X_J)_{\Gamma_t(\bar{p})} } \right \rangle=0,
\end{equation}
for all $X_J=X_{j_1} \wedge \ldots \wedge X_{j_m}$, with $1\le j_1<\cdots < j_m\le n$, such that $\deg(X_J)>\deg(M)$. Taking the derivative with respect to $t$ in equality \eqref{flow that preserves degree sub} and evaluating at $t=0$ we obtain the condition
\begin{equation*}
0=\langle e_1\wedge\ldots \wedge e_m,\nabla_{V(p)}X_J\rangle +\sum_{k=1}^m \langle e_1 \wedge \ldots \wedge \nabla_{e_k}  V \wedge \ldots \wedge e_m ,X_J\rangle 
\end{equation*}
for all $X_J$ such that $\deg(X_J)>\deg(M)$. In the above formula, $\escpr{\cdot,\cdot}$ indicates the scalar product in the space of $m$-vectors induced by the Riemannian metric $g$. The symbol $\nabla$ denotes, in the left summand, the Levi-Civita connection associated to $g$ and, in the right summand, the covariant derivative of vectors in $\mathfrak{X}(\bar{M},N)$ induced by $g$. Thus, if a variation preserves the degree then the associated variational vector field satisfies the above condition and we are led to the following definition.

\begin{definition}
\label{def:adm}
Given an immersion $\Phi:\bar{M}\to N$, a vector field $V\in \mathfrak{X}_0(\bar{M},N)$ is said to be \emph{admissible}  if it satisfies the system of first order PDEs
 \begin{equation}
 \label{System of PDEs for admissible}
 0=\langle e_1\wedge\ldots \wedge e_m,\nabla_{V(p)}X_J\rangle +\sum_{k=1}^m \langle e_1 \wedge \ldots \wedge\nabla_{e_k}  V \wedge \ldots \wedge e_m ,X_J\rangle 
\end{equation}
where $X_J=X_{j_1} \wedge \ldots \wedge X_{j_m}$,  $\deg(X_J)>d$ and $p\in M$. We denote by $\mathcal{A}_{\Phi}(\bar{M},N)$ the set of admissible vector fields.
\end{definition}

It is not difficult to check that the conditions given by \eqref{System of PDEs for admissible} are independent of the choice of the adapted basis.

Thus we are led naturally to a problem of integrability: given $V\in\mathfrak{X}_0(\bar{M},N)$ such that the first order condition \eqref{System of PDEs for admissible} holds, we ask whether an admissible variation whose associated variational vector field is $V$ exists. 

\begin{definition}
We say that an admissible vector field $V\in\mathfrak{X}_0(\bar{M},N)$ is \emph{integrable} if there exists an admissible variation such that the associated variational vector field is $V$.
\end{definition}

\begin{proposition}
\label{prop:normaladm}
Let $\Phi: \bar{M} \to N$ be an immersion into a graded manifold. Then a vector field $V \in \mathfrak{X}_0(\bar{M},N)$ is admissible if and only if  its normal component $V^{\perp}$ is admissible.
\end{proposition}

\begin{proof}
Since the Levi-Civita connection and the covariant derivative are additive we deduce that the admissibility condition \eqref{System of PDEs for admissible} is additive in $V$. We decompose $V=V^\top+V^{\perp}$ in its tangent $V^\top$ and normal $V^{\perp}$ components and observe that $V^\top$ is always admissible since the flow of $V^\top$ is an admissible variation leaving $\Phi(\bar{M})$ invariant with variational vector field $V^\top$. Hence, $V^{\perp}$ satisfies \eqref{System of PDEs for admissible} if and only if $V$  verifies \eqref{System of PDEs for admissible}.
\end{proof}

\section{The structure of the admissibility system of first order PDEs}
\label{sc:structure}
Let us consider an open set $U\subset N$ where a local adapted basis $(X_1,\ldots,X_n)$ is defined. We know that the simple $m$-vectors $X_J:=X_{j_1}\wedge\ldots\wedge X_{j_m}$ generate the space $\Lambda_m(U)$ of $m$-vectors. At a given point $p\in U$, its dimension is given by the formula
\[
\dim(\Lambda_m(U)_p)=\binom{n}{m}.
\]

Given two $m$-vectors $v,w\in\Lambda_m(U)_p$, it is easy to check that $\deg(v+w) \le \max\{\deg{v},\deg{w}\}$, and that $\deg{\la v}=\deg{v}$ when $\la\neq 0$ and $0$ otherwise. This implies that the set
\[
\Lambda_m^d(U)_p:=\{v\in\Lambda_m(U)_p: \deg{v}\le d\}
\]
is a vector subspace of $\Lambda_m(U)_p$. To compute its dimension we let $v_i:=(X_i)_p$ and we check that a basis of $\Lambda_m^d(U)_p$ is composed of the vectors
\[
v_{i_1}\wedge\ldots \wedge v_{i_m}\ \text{such that } \sum_{j=i_1}^{i_m}\deg(v_j)\le d.
\]
To get an $m$-vector in such a basis we pick any of the $k_1$ vectors in $\mathcal{H}^1_p\cap\{v_1,\ldots,v_n\}$ and, for $j=2,\ldots,s$, we pick any of the $k_j$ vectors on $(\mathcal{H}^j_p\smallsetminus \mathcal{H}^{j-1}_p)\cap \{v_1,\ldots,v_n\}$, so that
\begin{itemize}
\item $k_1+\cdots +k_s=m$, and
\item $1\cdot k_1+\cdots +s\cdot k_s\le d$.
\end{itemize}
So we conclude, taking $n_0=0$, that
\[
\dim(\Lambda_m^d(U)_p)=\sum_{ \substack{
            k_1+ \cdots + k_s=m,\\
            1 \cdot k_1+ \cdots + s \cdot k_s\le d}}
\bigg(\prod_{i=1}^s \binom{n_i-n_{i-1}}{k_i} \bigg).
\]

When we consider two simple $m$-vectors $v_{i_1}\wedge\ldots\wedge v_{i_m}$ and $v_{j_1}\wedge\ldots\wedge v_{j_m}$, their scalar product is $0$ or $\pm 1$, the latter case when, after reordering if necessary, we have $v_{i_k}=v_{j_k}$ for $k=1,\ldots,m$. This implies that the orthogonal subspace $\Lambda_m^d(U)_p^\perp$ of $\Lambda_m^d(U)_p$ in $\Lambda_m(U)_p$ is generated by the $m$-vectors
\[
v_{i_1}\wedge\ldots \wedge v_{i_m}\ \text{such that } \sum_{j=i_1}^{i_m}\deg(v_j)> d.
\]
Hence we have 
\begin{equation}
\label{dimension of m-vector of degree grater than d}
\dim(\Lambda_m^d(U)_p^\perp)=\sum_{ \substack{
            k_1+ \cdots + k_s=m,\\
            1 \cdot k_1+ \cdots + s \cdot k_s> d}}
\bigg(\prod_{i=1}^s \binom{n_i-n_{i-1}}{k_i} \bigg),
\end{equation}
with $n_0=0$. Since $N$ is equiregular, $\ell=\dim(\Lambda_m^d(U)_p^{\perp})$ is constant on $N$. Then we can choose an orthonormal basis $(X_{J_1},\ldots,X_{J_{\ell}})$ in $\Lambda_m^d(U)_p^{\perp}$ at each point $p \in U$.

\subsection{The admissibility system with respect to an adapted local basis}

In the same conditions as in the previous subsection, let $\ell=\dim (\Lambda_m^d(U)_p^{\perp})$ and $(X_{J_1},\ldots,X_{J_{\ell}})$ an orthonormal basis of $\Lambda_m^d(U)_p^{\perp}$. Any vector field $V\in\mathfrak{X}(\bar{M},N)$ can be expressed in the form
\[
V=\sum_{h=1}^n f_h X_h ,
\] 
where $f_1,\ldots,f_n  \in  C^{\infty}(\Phi^{-1}(U), \rr )$. We take $\bar{p}_0\in\Phi^{-1}(U)$ and, reducing $U$ if necessary, a local adapted basis $(E_i)_i$ of $T\bar{M}$ in $\Phi^{-1}(U)$. Hence the admissibility system~\eqref{System of PDEs for admissible} is equivalent to 
\begin{equation}
\label{eq:local system of PDEs}
  \sum_{j=1}^m \sum_{h=1}^n c_{i j h} \, E_j (f_h)  +  \sum_{h=1}^n \beta_{i h} \, f_h =0, \quad i=1,\ldots,\ell,
\end{equation}
where
\begin{equation}
\label{def:ccoef}
 c_{i j h}(\bar{p})= \escpr{e_1 \wedge \ldots \wedge \overset{(j)} {(X_h)_p} \wedge \ldots \wedge e_m, (X_{J_i})_p },
\end{equation}
and
\begin{equation}
\begin{split}
\label{def:bcoef}
\beta_{i h}(\bar{p})&=\escpr{e_1\wedge \ldots \wedge e_m,\nabla_{(X_h)_p} X_{J_i}}+
\\ &\qquad\qquad + 
\sum_{j=1}^m \escpr{e_1 \wedge \ldots \wedge \nabla_{e_j} X_h \wedge \ldots \wedge e_m, (X_{J_i})_p}\\
 &=\sum_{j=1}^m \escpr{e_1 \wedge \ldots \wedge [E_j,X_h] (p) \wedge \ldots \wedge e_m, (X_{J_i})_p}.
 \end{split}
\end{equation}
In the above equation we have extended the vector fields $E_i$ in a neighborhood of $p_0=\Phi(\bar{p}_0)$ in $N$, denoting them in the same way.

\begin{definition}
\label{def:rho}
Let $\tilde{m}_{\alpha}(p)$ be the dimension of $\tilde{\mh}_p^{\alpha}=T_p M \cap \mh^{\alpha}_p$, $\alpha\in \{1,\ldots,s\}$, where we consider the flag defined in \eqref{tangent flag of M}. Then we set 
\[
\iota_0(U)=\max_{p \in U}  \min_{1\le \alpha\le s} \{ \alpha  :  \tilde{m}_{\alpha} (p) \ne 0 \}.
\]
and
\begin{equation}
\label{def:nalpha}
\rho:=n_{\iota_0}=\dim(\mh^{\iota_0})\ge \dim(\mh^1)=n_1.
\end{equation}
\end{definition}

\begin{remark}
In the differential system \eqref{eq:local system of PDEs}, derivatives of the function $f_h$ appear only when some coefficient $c_{ijh}(\bar{p})$ is different from $0$. For fixed $h$, notice that $c_{i j h}(\bar{p})=0$, for all $i=1,\ldots,\ell$, $j=1,\ldots,m$ and $\bar{p}$ in $\Phi^{-1}(U)$ if and only if 
\[
\deg(e_1 \wdw \overset{(j)} {(X_h)_p} \wdw e_m)\le d,\quad \text{for all } 1\le j\le m, p \in\Phi^{-1}(U).
\]
This property is equivalent to
\[
\deg((X_h)_p)\le\deg(e_j), \text{for all } 1\le j\le m, p \in\Phi^{-1}(U).
\]
So we have $c_{ijh}=0$ in $\Phi^{-1}(U)$ for all $i,j$ if and only if $\deg(X_h)\le \iota_0(U)$.
\end{remark}

We write
\[
V=\sum_{h=1}^{\rho} g_h  X_h+ \sum_{r=\rho+1}^n f_r   X_r ,
\] 
so that the local system \eqref{eq:local system of PDEs} can be written as
\begin{equation}
\label{eq:local system of PDEs2}
  \sum_{j=1}^m \sum_{r=\rho+1}^n c_{i j r}  E_j (f_r)  +  \sum_{r=\rho+1}^n b_{i r} f_r + \sum_{h=1}^\rho a_{i h} g_h  =0,
\end{equation}
where $c_{i j r}$ is defined in \eqref{def:ccoef} and, for $1\le i\le\ell$,
\begin{equation}
\label{def:ab}
a_{ih}=\beta_{ih},\quad b_{ir}=\beta_ {ir},\ 1\le h\le\rho,\ \rho+1\le r\le n,
\end{equation}
where $\beta_{ij}$ is defined in \eqref{def:bcoef}. We denote by $B$ the $\ell \times (n-\rho)$ matrix whose entries are $b_{i r}$,  by $A$ the  $\ell \times \rho$ whose entries are $a_{i h}$ and for $j=1,\ldots,m$  we denote by $C_j$ the  $\ell \times(n- \rho )$ matrix 
 $C_j=(c_{ijh})_{h=\rho+1,\ldots ,n}^{i=1,\ldots,\ell}$.
Setting 
\begin{equation}
\label{def:FG}
F=\begin{pmatrix} f_{\rho+1} \\ \vdots \\ f_n \end{pmatrix}, \quad G=\begin{pmatrix} g_1 \\ \vdots \\ g_{\rho} \end{pmatrix}
\end{equation}
the admissibility system \eqref{eq:local system of PDEs} is given by
\begin{equation}
\sum_{j=1}^m C_j E_j(F)+ BF+ AG=0.
\label{eq:local system of PDEs3}
\end{equation}

\subsection{Independence on the metric}
\label{ssc:independence}
Let $g$ and $\tilde{g}$ be two Riemannian metrices on $N$ and $(X_i)$  be orthonormal adapted basis with respect to $g$  and $(Y_i)$ with respect to $\tilde{g}$. Clearly we have 
\[
Y_i=\sum_{j=1}^n d_{ji} X_j,
\]
for some square invertible matrix $D=(d_{ji})_{j=1,\ldots,n}^{i=1,\ldots,n}$ of order $n$. Since $(X_i)$ and $(Y_i)$ are adapted basis, $D$ is a block matrix
\[
D=\begin{pmatrix}
D_{1 1} & D_{1 2}& D_{1 3} & \ldots & D_{1 s} \\
0 & D_{2 2}& D_{2 3} &   \ldots         & D_{2 s}\\
0& 0 & D_{3 3}& \ldots &D_{3 s}\\
0& 0  & 0& \ddots& \vdots\\

0 & 0    &0 &  0& D_{s s}\\
\end{pmatrix},
\]
where $D_{i i}$ for $i=1,\ldots,s$ are square matrices of orders $n_i$. Let $\rho$ be the integer defined in \eqref{def:rho}, then we define $D_h=(d_{ji})_{i,j=1,\ldots,\rho}$, $D_v=(d_{ji})_{i,j=\rho+1,\ldots,n}$ and $D_{hv}=(d_{ji})_{j=1,\ldots,\rho}^{i=\rho+1,\ldots,n}$. Let us express $V$ as a linear combination of $(Y_i)$
\[
V=\sum_{h=1}^{\rho} \tilde{g}_h Y_h+\sum_{r=\rho+1}^n \tilde{f}_r Y_r,
\]
then we set
\[
\tilde{F}=\begin{pmatrix} \tilde{f}_{\rho+1} \\ \vdots \\ \tilde{f}_n \end{pmatrix}, \quad \tilde{G}=\begin{pmatrix} \tilde{g}_1 \\ \vdots \\ \tilde{g}_{\rho} \end{pmatrix}
\]
and $F$ and $G$ as in \eqref{def:FG}.

Given $I=(i_1,\ldots, i_m)$ with $i_1 < \ldots < i_m$, we have
\begin{align*}
Y_I&=Y_{i_1} \wdw Y_{i_m}=\sum_{j_1=1}^n \cdots \sum_{j_m=1}^n  d_{j_1 i_1} \cdots d_{j_m i_m} X_{j_1} \wdw  X_{j_m}\\
 &= \sum_{j_1< \ldots < j_m} \lambda^{j_1 \ldots j_m}_{i_1 \ldots i_m} X_{j_1} \wdw  X_{j_m}= \sum_{J} \lambda_{J I} X_J.
\end{align*}
Since the adapted change of basis preserves the degree of the $m$-vectors, the square matrix $\mathbf{\Lambda}=(\lambda_{J I}  )$ of order $\binom{n}{m}$ acting on the $m$-vector  is  given by 
\begin{equation}
\mathbf{\Lambda}=\begin{pmatrix}
\mathbf{\Lambda}_h & \mathbf{\Lambda}_{hv}\\
0 & \mathbf{\Lambda}_v,
\end{pmatrix}
\end{equation} 
where $\mathbf{\Lambda}_h$ and $\mathbf{\Lambda}_v$ are square matrices of order $\binom{n}{m}-\ell$ and $\ell$ respectively and $\mathbf{\Lambda}_{hv}$ is a matrix of order $\left(\binom{n}{m}-\ell \right) \times \ell$. Moreover the matrix $\mathbf{\Lambda}$ is invertible since both $\{X_J\}$ and $\{Y_I\}$ are basis of the vector space of $m$-vectors.
\begin{remark}
One can easily check that the inverse of $\mathbf{\Lambda}$ is given by the block matrix 
\[
\mathbf{\Lambda}^{-1}=\begin{pmatrix}
\mathbf{\Lambda}_h^{-1} & - \mathbf{\Lambda}_h^{-1} \, \mathbf{\Lambda}_{hv} \mathbf{\Lambda}_{v}^{-1}\\
0 & \mathbf{\Lambda}_v^{-1}
\end{pmatrix}.
\]
Setting $\mathbf{\tilde{G}}=(\tilde{g}(X_I,X_J))$ we have 
\[
\mathbf{\tilde{G}}=\begin{pmatrix}
\mathbf{\tilde{G}}_h & \mathbf{\tilde{G}}_{hv}\\
(\mathbf{\tilde{G}}_{hv})^t & \mathbf{\tilde{G}}_v
\end{pmatrix}=(\LMD^{-1})^t (\LMD^{-1}).
\]
Thus it follows  
\begin{align*}
\GT_v&= (\LMD_v^{-1})^t \LMD_v^{-1} + (\LMD_v^{-1})^t \LMD_{hv}^t ( \LMD_h^{-1})^t \, \LMD_h^{-1} \LMD_{hv} \LMD_v^{-1},\\
\GT_{hv}&=-  (\LMD_h^{-1})^t\LMD_h^{-1} \LMD_{hv} \LMD_v^{-1},\\
\GT_{h}&= (\LMD_h^{-1})^t  \LMD_h^{-1} .
\end{align*}
\label{rk:MG}
\end{remark}

Let $\tilde{A}$ be the associated matrix 
\begin{align*}
\tilde{A}=\Big( \tilde{g} \Big( Y_{J_i} , \sum_{j=1}^m E_1 \wedge \ldots \wedge [E_j, Y_h](p) \wedge \ldots \wedge E_m \Big) \Big)^{h=1,\ldots,\rho}_{i=1,\ldots,\ell}.
\end{align*}
Setting
\[
\omega_{J r}=\sum_{j=1}^mg( X_J, E_1 \wdw [E_j, X_r] \wdw E_m),
\]
and $\Omega=\begin{pmatrix}\Omega_h & \Omega_v\end{pmatrix}=(\omega_{J r})_{\deg(J)\le d}^{r=1,\ldots,n}$,
a straightforward computation shows
\begin{equation*}
\begin{aligned}
\tilde{A}=& (\LMD_{hv})^t \Bigg(\GT_{h} \, \Omega_h \, D_h+  \GT_{hv} \, A \, D_h + \GT_{h} \,  \sum_{j=1}^m C_j E_j( D_h ) \Bigg)\\
&+ (\LMD_v)^t \Bigg( (\GT_{hv})^t \, \Omega_h D_h+ \GT_v \, A \, D_h+(\GT_{hv})^t \sum_{j=1}^m C_j E_j( D_h )  \Bigg)
\end{aligned}
\end{equation*}
By Remark \ref{rk:MG} we obtain 
\begin{equation}
\label{eq:tildeA}
\begin{aligned}
\tilde{A}=& (\LMD_{hv})^t \Big((\LMD_h^{-1})^t \LMD_h^{-1} \,( \Omega_h \, D_h + \sum_{j=1}^m C_j E_j( D_h ) )\\
& \quad  -(\LMD_h^{-1})^t \LMD_h^{-1} \LMD_{hv} \LMD_v^{-1} \, A \, D_h   \Big)\\
&- \Big(  \LMD_{hv}^t (\LMD_h^{-1})^t \LMD_h^{-1} \, (\Omega_h D_h + \sum_{j=1}^m C_j E_j( D_h ))\Big)\\
&+\left(  \LMD_v^{-1} +  \LMD_{hv}^t ( \LMD_h^{-1})^t \, \LMD_h^{-1} \LMD_{hv} \LMD_v^{-1} \right) \, A \, D_h\\
=& \LMD_v^{-1} \, A \, D_h.
\end{aligned}
\end{equation}
Preliminary we notice that if $h=1,\ldots,\rho$ we have 
\begin{equation}
\begin{aligned}
 \tilde{c}_{i j h}&= \tilde{g} (Y_{J_i} , E_1 \wedge \ldots \wedge \overset{(j)} {Y_h} \wedge \ldots \wedge E_m)\\
 &=\sum_I \sum_{deg(J)\le d} \sum_{k=1}^{\rho} \lambda_{I J_i} \, \tilde{g}(X_I, X_J) c_{J j k}\, d_{kh}\\
 &=\sum_{\deg(I)\le d} \sum_{\deg(J)\le d} \sum_{k=1}^{\rho} \lambda_{I J_i} \, \tilde{g}(X_I, X_J) c_{J j k}\, d_{kh}+\\
& \quad+ \sum_{\deg(I)> d} \sum_{\deg(J)\le d} \sum_{k=1}^{\rho} \lambda_{I J_i} \, \tilde{g}(X_I, X_J) c_{J j k}\, d_{kh}.
\end{aligned}
\label{eq:c0}
\end{equation}
Therefore, setting
\[
\tilde{C}^{H}_j=\Big( \tilde{g} (Y_{J} , E_1 \wedge \ldots \wedge \overset{(j)} {Y_h} \wedge \ldots \wedge E_m) \Big)_{\deg(J)\le d}^{h=1,\ldots,\rho}
\]
and 
\[
\tilde{C}^{0}_j=\Big( \tilde{g} (Y_{J_i} , E_1 \wedge \ldots \wedge \overset{(j)} {Y_h} \wedge \ldots \wedge E_m) \Big)_{i=1,\ldots,\ell}^{h=1,\ldots,\rho},
\]
by \eqref{eq:c0} we gain
\[
\tilde{C}_j^{0}= (\LMD_{hv}^t \GT_h+ \LMD_v^t (\GT_{hv})^t )( C_j^{H} D_h)=0.
\]
Let $\tilde{C}_j$ be the associated matrix 
\[
\tilde{C}_j=\Big( \tilde{g} (Y_{J_i} , E_1 \wedge \ldots \wedge \overset{(j)} {Y_h} \wedge \ldots \wedge E_m) \Big)_{i=1,\ldots,\ell}^{h=\rho+1,\ldots,n}.
\]
Setting 
\[
\tilde{C}^{HV}_j=\Big( \tilde{g} (Y_{J} , E_1 \wedge \ldots \wedge \overset{(j)} {Y_h} \wedge \ldots \wedge E_m) \Big)_{\deg(J)\le d}^{h=\rho+1,\ldots,n},
\]
it is immediate to obtain the following equality
\begin{equation}
\begin{aligned}
\tilde{C}_j=& (\LMD_{hv})^t \left(\GT_{h} (C_j^H D_{hv}+C_j^{HV} D_v)+ \GT_{hv} C_j D_v\right)\\
&+ (\LMD_v)^t \left( (\GT_{hv})^t (C_j^H D_{hv}+C_j^{HV} D_v)+ \GT_v C_j D_v \right)\\
=&\LMD_v^{-1} C_j D_v .
\end{aligned}
\label{eq:tildeC}
\end{equation}
Let $\tilde{B}$ be the associated matrix 
\begin{align*}
\tilde{B}=\Big( \tilde{g} \Big( Y_{J_i} , \sum_{j=1}^m E_1 \wedge \ldots \wedge [E_j, Y_h] \wedge \ldots \wedge E_m \Big) \Big)^{h=\rho+1,\ldots,n}_{i=1,\ldots,\ell}.
\end{align*}
A straightforward computation shows
\begin{equation*}
\begin{aligned}
\tilde{B}=& (\LMD_{hv})^t \Big( \GT_h ( \Omega_h \, D_{hv}+\Omega_v D_v + \sum_{j=1}^m C_j^H E_j(D_{hv})+  C_j^{HV} E_j( D_h ) )\\
&+ \GT_{hv}( A D_{hv}+ B D_v+ \sum_{j=1}^m C_j E_j(D_v) )  \Big)\\
&+ (\LMD_v)^t \Big( \GT_{hv}^t ( \Omega_h \, D_{hv}+\Omega_v D_v + \sum_{j=1}^m C_j^H E_j(D_{hv})+  C_j^{HV} E_j( D_h ) )\\
&+ \GT_{v}( A D_{hv}+ B D_v+ \sum_{j=1}^m C_j E_j(D_v) )  \Big)\\
\end{aligned}
\end{equation*}
By Remark \ref{rk:MG} we obtain 
\begin{equation}
\label{eq:tildeB}
\tilde{B}=  \LMD_v^{-1} \, A \, D_{hv} + \LMD_v^{-1} B  D_v +\sum_{j=1}^m \LMD_v^{-1} C_j E_j (D_v) .
\end{equation}
Finally, we have 
$
G=D_h \tilde{G}+ D_{hv} \tilde{F}
$
and $F=D_v \tilde{F}$.

\begin{proposition}
Let $g$ and $\tilde{g}$ be two different metrics, then a vector fields $V$ is admissible w.r.t. $g$ if and only if $V$ is admissible w.r.t. $\tilde{g}$.
\end{proposition}
\begin{proof}
We remind that an admissible vector field 
$$V=\sum_{i=1}^{\rho} g_i X_i + \sum_{i=\rho+1}^n f_i X_i$$
w.r.t. $g$ satisfies 
\begin{equation}
\label{eq:admprop}
\sum_{j=1}^m C_j E_j(F) + BF +AG=0.
\end{equation}
By \eqref{eq:tildeA}, \eqref{eq:tildeB} and \eqref{eq:tildeC} we have 
\begin{equation}
\label{eq:admproptilde}
\begin{aligned}
&\sum_{j=1}^m \tilde{C}_j E_j(\tilde{F}) + \tilde{B}\tilde{F} +\tilde{A}\tilde{G}=\LMD_v^{-1} \Big( \sum_{j=1}^m C_j (D_v E_j(\tilde{F} )+ E_j(D_v) \tilde{F}) \\
& \quad +A \, D_{hv} \tilde{F} + A \, D_h \tilde{G}+ B D_v \tilde{F} \Big)=\LMD_v^{-1} \Big(\sum_{j=1}^m C_j E_j(F) + BF +AG\Big)
\end{aligned}
\end{equation}
In the previous equation we used that  $G=D_h \tilde{G}+ D_{hv} \tilde{F}$, $F=D_v \tilde{F}$ and 
\[
E_j(D_v) D_v^{-1} + D_v E_j(D_v^{-1})=0,
\]
for all $j=1,\ldots,m$, that follows by $D_v D_v^{-1}=I_{n-\rho}$. Then the admissibility system \eqref{eq:admprop} w.r.t. $g$  is equal to zero if and only if the admissibility system  \eqref{eq:admproptilde}  w.r.t. $\tilde{g}$.
\end{proof}

\begin{remark}
\label{rk:orthchange}
When the metric $g$ is fixed and $(X_i)$ and $(Y_i)$ are orthonormal adapted basis w.r.t $g$, the matrix $D$ is a block diagonal matrix given by
\[
D=\begin{pmatrix}
D_h & 0 \\
0 & D_v
\end{pmatrix},
\]
where $D_h$ and $D_v$ are square orthogonal matrices of orders $\rho$ and $(n-\rho)$, respectively. From equations \eqref{eq:tildeA}, \eqref{eq:tildeB}, \eqref{eq:tildeC} it is immediate to obtain the following equalities
\begin{equation}
\label{eq:tilde}
\begin{split}
\tilde{F}&=D_v^{-1} F,
\\
\tilde{G}&=D_h^{-1} G,
\\
\tilde{A}&=\mathbf{\Lambda}_v^{-1} \  A  \ D_h,
\\
\tilde{B}&=\mathbf{\Lambda}_v^{-1} B D_v+ \sum_{j=1}^m \mathbf{\Lambda}_{v}^{-1} C_j E_j (D_v),
\\
\tilde{C}_j&=\mathbf{\Lambda}_{v}^{-1}  C_j D_v.
\end{split}
\end{equation}
\end{remark}

\subsection{The admissibility system with respect to the intrinsic basis of the normal space}
\label{ss:intriniscadmsy}

Let $\ell$ be the dimension of $\Lambda_m^d(U)_p^{\perp}$ and $(X_{J_1},\ldots,X_{J_{\ell}})$ an orthonormal basis of simple $m$-vector fields. Let $\bar{p}_0$ be a point in $\bar{M}$ and $\Phi(\bar{p}_0)=p_0 $. Let $e_1,\ldots,e_m$ be an adapted basis of $T_{p_0} M$ that we extend to adapted vector fields  $E_1,\ldots,E_m$  tangent to $M$ on $U$.  Let $v_{m+1},\ldots,v_n$ be a basis of $(T_{p_0} M)^{\perp}$ that we extend to vector fields $V_{m+1}, \ldots,V_n$  normal to $M$ on $U$, where we possibly reduced the  neighborhood $U$ of $p_0$ in $N$.
  Then any vector field in $\mathfrak{X}(\Phi^{-1}(U) ,N)$ is given by 
\[
V=\sum_{j=1}^m \psi_j E_j+\sum_{h=m+1}^n \psi_h V_h ,
\] 
where $\psi_1,\ldots, \psi_n  \in  C^{r}( \Phi^{-1}(U) , \rr )$. By Proposition~\ref{prop:normaladm}  we deduce that $V$ is admissible if and only if $V^{\perp}=\sum_{h=m+1}^n \psi_h V_h$ is admissible. Hence we obtain that the system~\eqref{System of PDEs for admissible} is equivalent to 
\begin{equation}
\label{eq:normal system of PDEs}
  \sum_{j=1}^m \sum_{h=m+1}^n \xi_{i j h}  E_j (\psi_h)  +  \sum_{h=m+1}^n \hat{\beta}_{i h} \psi_h =0, \quad i=1,\ldots,\ell,
\end{equation}
where
\begin{equation}
\label{def:xicoef}
 \xi_{i j h}(\bar{p})= \escpr{e_1 \wedge \ldots \wedge \overset{(j)} {v_h} \wedge \ldots \wedge e_m, (X_{J_i})_p }
\end{equation}
and
\begin{equation}
\begin{aligned}
\label{def:betacoef}
\hat{\beta}_{i h}(\bar{p})&=\escpr{e_1\wedge \ldots \wedge e_m,\nabla_{v_h} X_{J_i}}+
\\ &\quad\qquad + \sum_{j=1}^m \escpr{e_1 \wedge \ldots \wedge \nabla_{e_j} V_h \wedge \ldots \wedge e_m, (X_{J_i})_p}\\
 &=\sum_{j=1}^m \escpr{e_1 \wedge \ldots \wedge [E_j, V_h] (p) \wedge \ldots \wedge e_m, (X_{J_i})_p}.
\end{aligned}
\end{equation}

\begin{definition}
\label{def:k}
Let $\iota_0(U)$ be the integer defined in \ref{def:rho}. Then we set $k:=n_{\iota_0}-\tilde{m}_{\iota_0}$.
\end{definition}
Assume that $k\ge1$, and write 
\[
V^{\perp}=\sum_{h=m+1}^{m+k} \phi_h \, V_h+ \sum_{r=m+k+1}^n \psi_r \,  V_r ,
\] 
and the local system \eqref{eq:normal system of PDEs} is equivalent to 
\begin{equation}
\label{eq:local system of PDEs2 normal}
  \sum_{j=1}^m \sum_{r=\rho+1}^n \xi_{i j r} \, E_j (\psi_r)  +  \sum_{r=\rho+1}^n \beta_{i r} \, \psi_r + \sum_{h=m+1}^{m+k} \alpha_{i h} \, \phi_h  =0,
\end{equation}
where $\xi_{i j r}$ is defined in \eqref{def:xicoef}  and, for $1\le i\le\ell$,
\begin{equation}
\label{def:albe}
\alpha_{ih}=\hat{\beta}_{ih},\quad \beta_{ir}=\hat{\beta}_ {ir},\ m+1\le h\le m+k,\ m+k+1\le r\le n.
\end{equation}
We denote by $B^{\perp}$ the $\ell \times (n-m-k)$ matrix whose entries are $\beta_{i r}$,  by $A^{\perp}$ the  $\ell \times k$ whose entries are $\alpha_{i h}$ and for every $j=1, \cdots m$ by $C_j^{\perp}$ the $\ell \times (n- m-k ) $ matrix with entries $(\xi_{ijh})_{h=m+k+1,\ldots ,n}^{i=1,\ldots,\ell} $
Setting 
\begin{equation}
\label{def:FGperp}
F^{\perp}=\begin{pmatrix} \psi_{m+k+1} \\ \vdots \\ \psi_n \end{pmatrix}, \quad G^{\perp}=\begin{pmatrix} \phi_{m+1} \\ \vdots \\ \phi_{m+k} \end{pmatrix}
\end{equation}
the admissibility system \eqref{eq:local system of PDEs} is given 
\begin{equation}
\sum_{j=1}^m C^{\perp}_j  E_j(F^{\perp})+ B^{\perp}F^{\perp}+ A^{\perp} G^{\perp}=0.
\label{eq:normal system of PDEs3}
\end{equation}

\begin{remark}
\label{rk:Atop}
We can define the matrices $A^\top$, $B^\top$, $C^\top$ with respect to the tangent projection $V^\top$ in a similar way to the matrices $A^\bot$, $B^\bot$, $C^\bot$. First of all we notice that the entries 
\[
 \xi^{\top}_{i j \nu}(\bar{p})= \escpr{e_1 \wedge \ldots \wedge \overset{(j)} {e_{\nu}} \wedge \ldots \wedge e_m, (X_{J_i})_p }
\]
for $i=1,\ldots,\ell$ and $j,\nu=1,\ldots,m$ are all equal to zero. Therefore the matrices $C^{\top}$ and $B^{\top}$ are equal to zero. On the other hand, $A^{\top}$ is the $(\ell \times m)$-matrix whose entries are given by
\[
\alpha^{\top}_{i \nu}(\bar{p})=\sum_{j=1}^m \escpr{e_1 \wedge \ldots \wedge [E_j, E_{\nu}] (p) \wedge \ldots \wedge e_m, (X_{J_i})_p}
\]
for $i=1, \ldots,\ell$ and $\nu=1,\ldots,m$. Frobenius Theorem implies that the Lie brackets $[E_j,E_\nu]$ are all tangent to $M$ for $j,\nu=1,\ldots,m$, and so all the entries of $A^{\top}$ are equal to zero.
\end{remark}

\section{Integrability of admissible vector fields}
\label{sc:integrability}
In general, given an admissible vector field $V$, the existence of an admissible variation with associated variational vector field $V$ is not guaranteed. The next definition is a sufficient condition to ensure the integrability of admissible vector fields.

\begin{definition}
\label{def:regular}
Let $\Phi:\bar{M}\to N$ be an immersion of degree $d$ of an $m$-dimensional manifold into a graded manifold endowed with a Riemannian metric $g$. Let $\ell=\dim(\Lambda_m^d(U)_q^\perp)$ for all $q\in N$ and $\rho=n_{\iota_0}$ set in \eqref{def:rho}. When $\rho \ge \ell$ we say that $\Phi$ is \emph{strongly regular} at $\bar{p} \in \bar{M}$ if 
\[
\text{rank} (A(\bar{p}))=\ell ,
\]
where $A$ is the matrix appearing in the admissibility system \eqref{eq:local system of PDEs3}.
\end{definition}

The rank of $A$ is independent of the local adapted basis chosen to compute the admissibility system \eqref{eq:local system of PDEs3} because of equations \eqref{eq:tilde}. Next we prove that strong regularity is a sufficient condition to ensure local integrability of admissible vector fields.

\begin{theorem}
\label{teor:locint}
Let $\Phi:\bar{M}\to N$ be a smooth immersion of an $m$-dimensional manifold   into an equiregular graded manifold $N$ endowed with a Riemannian metric $g$. Assume that the immersion  $\Phi$ of degree $d$ is strongly regular at $\bar{p}$. Then there exists an open neighborhood $W_{\bar{p}}$ of $\bar{p}$ such every admissible vector field $V$ with compact support on $W_{\bar{p}}$ is integrable.
\end{theorem}

\begin{proof}
Let $p=\Phi(\bar{p})$. First of all we consider an open neighborhood $U_{p} \subset N$ of $p$ such that an adapted orthonormal frame $(X_1,\ldots, X_n)$ is well defined.
Since $\Phi$ is strongly regular at $\bar{p}$ there exist indexes $h_1,\ldots, h_{\ell}$ in $\{1,\ldots,\rho\}$ such that the submatrix 
\[
\hat{A} (\bar{p} )=\begin{pmatrix}
a_{1h_1} (\bar{p} )& \cdots & a_{1 h_{\ell}}(\bar{p} )\\
\vdots & \ddots & \vdots\\
a_{\ell  h_1}(\bar{p} )& \cdots & a_{\ell h_{\ell}}(\bar{p} )
\end{pmatrix}
\]
is invertible. By a continuity argument there exists an open neighborhood $W_{\bar{p} } \subset \Phi^{-1}( U_p )$ such that $\det(\hat{A} (\bar{q}))\ne0$ for each $\bar{q} \in W_{\bar{p} }$.  

We can rewrite the system \eqref{eq:local system of PDEs3} in the form
\begin{equation}
\begin{pmatrix}
g_{h_1}\\
\vdots\\
g_{h_{\ell}}
\end{pmatrix}= -\hat{A}^{-1} 
\left( \sum_{j=1}^m C_j E_j (F) + B F +\tilde{A} \begin{pmatrix}
g_{i_1}\\
\vdots\\
g_{i_{\rho-\ell}}
\end{pmatrix} \right),
\label{eq:adm1}
\end{equation}
where $i_1, \ldots, i_{\rho-\ell}$ are the indexes of the columns of $A$ that do not appear in $\hat{A}$ and  $\tilde{A}$ is the $\ell \times (\rho-\ell)$ matrix given by the columns $i_1, \ldots, i_{\rho-\ell}$ of $A$. The vectors $(E_i)_i$ form an orthonormal basis of $T\bar{M}$ near $\bar{p}$.

On the neighborhood $W_{\bar{p}}$ we define the following spaces

\begin{enumerate}
\item $\mathfrak{X}_0^r(W_{\bar{p}},N)$, $r\ge0$ is the set of $C^r$ vector fields compactly supported on $W_{\bar{p}}$ taking values in $TN$.
\item $\mathcal{A}_0^r(W_{\bar{p}},N)=\{Y \in \mathfrak{X}_0^r(W_{\bar{p}},N)  :  Y=\sum_{s=1}^{\rho} g_{s} X_{s}\}$.
\item $\mathcal{A}_{1, 0}^r(W_{\bar{p}},N)=\{Y\in \mathcal{A}_0^r(W_{\bar{p}},N)  :  Y=\sum_{i=1}^{\ell} g_{h_i} X_{h_i}  \}.$
\item $\mathcal{A}_{2,0}^r(W_{\bar{p}},N)=\{Y \in \mathcal{A}_0^r(W_{\bar{p}},N) : \escpr{Y,X}=0 \ \forall \ X \in  \mathcal{A}_{1,0}^r(W_{\bar{p}},N)\}$.
\item $\mathcal{V}_0^r(W_{\bar{p}} ,N)=\{Y\in \mathfrak{X}^r(W_{\bar{p}},N) :  \escpr{ Y,X}=0 \ \forall X \in \mathcal{A}_0^r(W_{\bar{p}},N) \}
=\mathcal{A}_0^r(W_{\bar{p}},N)^{\perp}.
$
\item $\Lambda_0^r(W_{\bar{p}},N)=\{ \sum_{i=1}^{\ell} f_i X_{J_i}  :  f_i \in C_0^r(W_{\bar{p}}) \}.$
\end{enumerate} 

Given $r\ge 1$, we set
\[
E:=\mathcal{A}_{2,0}^{r-1}(W_{\bar{p}},N) \times \mathcal{V}_0^r(W_{\bar{p}} ,N) ,
\]
and consider the map
\begin{equation}
\label{eq:defG}
\mathcal{G}: E \times  \mathcal{A}_{1,0}^{r-1}(W_{\bar{p}},N) \to E \times \Lambda_0^{r-1}(W_{\bar{p}},N) ,
\end{equation}
defined by
\[
\mathcal{G}(Y_1,Y_2,Y_3)=(Y_1,Y_2,\mathcal{F}(Y_1+Y_2+Y_3)),
\]
where $\Pi_v$ is the projection in the space of $m$-forms with compact support in $W_{\bar{p}}$ onto $\Lambda^r(W_{\bar{p}},N)$, and
\[
\mathcal{F}(Y)=\Pi_v\left( d\Gamma(Y)(e_1) \wedge \ldots \wedge d\Gamma(Y)(e_m)\right),
\]
where $\Gamma(Y)(p)=\exp_{\Phi(p)}(Y_p)$. Observe that $\mathcal{F}(Y)=0$ if and only if the submanifold $\Gamma(Y)$ has degree less or equal to $d$. We consider on each space the corresponding $||\cdot||_r$ or $||\cdot||_{r-1}$ norm, and a product norm.

Then
\[
D\mathcal{G}(0,0,0)(Y_1,Y_2,Y_3)=(Y_1,Y_2,D\mathcal{F}(0)(Y_1+Y_2+Y_3)),
\]
where we write in coordinates
\[
Y_1=\sum_{t=1}^{\rho-\ell} g_{i_t} \, X_{i_t}, \quad Y_2=\sum_{i=1}^{\ell} g_{h_i} \, X_{h_i}, \quad \text{and} \quad Y_3=\sum_{r=\rho+1}^{n} f_{r} \, X_{r}.
\]
Following the same argument we used in Section \ref{sc:admissible}, taking the derivative at $t=0$ of \eqref{flow that preserves degree sub}, we deduce that the differential $D\mathcal{F}(0)Y$ is given by
\[
D\mathcal{F}(0)Y= \sum_{i=1}^{\ell}\Bigg(\sum_{j=1}^m \sum_{r=\rho+1}^n c_{i j r}  E_j (f_r)  +  \sum_{r=\rho+1}^n b_{i r} f_r + \sum_{h=1}^\rho a_{i h} g_h  \Bigg) X_{J_i}.
\] 

Oberve that $D\mathcal{F}(0)Y=0$ if and only if $Y$ is an admissible vector field, namely $Y$ solves \eqref{eq:adm1}.
\\

Our objective now is to prove that the map $D\mathcal{G}(0, 0,0)$ is an isomorphism of Banach spaces.

Indeed suppose that $D\mathcal{G}(0, 0,0)(Y_1,Y_2,Y_3)=(0,0,0)$. This implies that $Y_1$ and $Y_2$ are equal zero. By the admissible equation \eqref{eq:adm1} we have that also $Y_3$ is equal to zero, then $D\mathcal{G}(0, 0,0)$ is injective. Then fix $(Z_1, Z_2, Z_3)$, where $Z_1 \in \mathcal{A}_{2,0}^{r-1}(W_{\bar{p}},N)$, $ Z_2 \in  \mathcal{V}_0^r(W_{\bar{p}},N) $, $ Z_3 \in \Lambda_0^{r-1}(W_{\bar{p}},N) $ we seek $Y_1,Y_2,Y_3$ such that $D\mathcal{G}(0, 0,0)(Y_1,Y_2,Y_3)=(Z_1, Z_2, Z_3)$. We notice that $D\mathcal{F}(0)(Y_1+Y_2+Y_3)=Z_3$ is equivalent to
\begin{equation*}
\left(\begin{array}{c}
z_{1}\\
\vdots\\
z_{\ell}
\end{array}\right)= 
\left( \sum_{j=1}^m C_j \, E_j (F) +  B F +  \tilde{A} \left(\begin{array}{c}
g_{i_1}\\
\vdots \\
g_{i_{\rho-\ell}}
\end{array}\right) + \hat{A} \left(\begin{array}{c}
g_{h_1}\\
\vdots\\
g_{h_{\ell}}
\end{array}\right)  \right),
\end{equation*} 
where with an abuse of notation we identify $Z_3=\sum_{i=1}^{\ell} z_i \  X_{J_i}$ and $\sum_{i=1}^{\ell} z_i \  X_{h_i}$. Since $\hat{A}$ is invertible we have the following system
\begin{equation}
\left(\begin{array}{c}
g_{h_1}\\
\vdots\\
g_{h_{\ell}}
\end{array}\right)= -\hat{A}^{-1} 
\left( \sum_{j=1}^m C_j \, E_j (F) +  B F +  \tilde{A} \left(\begin{array}{c}
g_{i_1}\\
\vdots \\
g_{i_{\rho-\ell}}
\end{array}\right)  + \left(\begin{array}{c}
z_{1}\\
\vdots\\
z_{\ell}
\end{array}\right)  \right).
\label{eq:adm2}
\end{equation}
Clearly $Y_1=Z_1$ fixes $g_{i_1},\ldots, g_{i_{\rho-\ell}}$ in \eqref{eq:adm2}, and $Y_2=Z_2$ fixes the first and second term of the right hand side in \eqref{eq:adm2}. Since the right side terms are given we have determined  $Y_3 $, i.e. $g_{h_1},\ldots,g_{h_{\ell}}$,  such that $Y_3$ solves  \eqref{eq:adm2}. Therefore $D\mathcal{G}(0, 0,0)$ is surjective. Thus we have proved that $D\mathcal{G}(0,0,0)$ is a bijection.

Let us prove now that $D\mathcal{G}(0,0,0)$ is a continuous and open map. Letting $D\mathcal{G}(0,0,0)(Y_1,Y_2,Y_3)=(Z_1,Z_2,Z_3)$, we first notice $D\mathcal{G}(0, 0,0)$ is a continuous map since identity maps are continuous and, by \eqref{eq:adm2}, there exists a constant $K$ such that 
\begin{align*}
\| Z_3 \|_{r-1} &\le K \Bigg( \sum_{j=1}^m \| \nabla_j Y_2 \|_{r-1}+ \|Y_2\|_{r-1}+ \|Y_1\|_{r-1} + \|Y_3\|_{r-1}\Bigg)\\
                      &\le K( \| Y_2\|_{r} + \|Y_1\|_{r-1} + \| Y_3\|_{r-1}).
\end{align*}
Moreover, $D\mathcal{G}(0, 0,0)$ is an open map since we have 
\begin{align*}
\| Y_3 \|_{r-1} &\le K \Bigg(  \sum_{j=1}^m \| \nabla_j Z_2 \|_{r-1}+ \|Z_2\|_{r-1}+ \|Z_1\|_{r-1} + \|Z_3\|_{r-1}\Bigg)\\
                      &\le K ( \| Z_2\|_{r} + \|Z_1\|_{r-1} + \| Z_3\|_{r-1}).
\end{align*}
This implies that $D\mathcal{G}(0,0,0)$ is an isomorphism pf Banach spaces.
\\

Let now us consider an admissible vector field  $V$ with compact support on $W_p$. We consider the map
\[
\tilde{\mathcal{G}}:(-\eps,\eps) \times E \times  \mathcal{A}_{0,1}^{r-1}(W_{\bar{p}},N) \to E \times \Lambda_0^{r-1}(W_{\bar{p}},N),
\]
defined by
\[
\tilde{\mathcal{G}}(s,Y_1,Y_3,Y_2)=(Y_1,\mathcal{F}(sV+Y_1+Y_3+Y_2)).
\]
The map $\tilde{\mathcal{G}}$ is continuous with respect to the product norms (on each factor we put the natural norm, the Euclidean one on the intervals and $||\cdot||_r$ and $||\cdot||_{r-1}$ in the spaces of vectors on $\Phi(\bar{M})$). Moreover
\[
\tilde{\mathcal{G}}(0,0,0,0)=(0,0),
\]
since $\Phi$ has degree $d$. Denoting by $D_Y$ the differential with respect to the last three variables of $\tilde{G}$ we have that
\[
D_{Y}\tilde{\mathcal{G}}(0,0,0,0)(Y_1,Y_2,Y_3)=D\mathcal{G}(0,0,0)(Y_1,Y_2,Y_3)
\]
is a linear isomorphism. We can apply the Implicit Function Theorem to obtain unique maps
\begin{equation}
\label{def:y3}
\begin{split}
&Y_1:(-\eps,\eps) \to \mathcal{A}_{0,2}^{r-1}(W_{\bar{p}},N), 
\\ &Y_2:(-\eps,\eps) \to \mathcal{V}_0^{r}(W_{\bar{p}},N),
\\ &Y_3:(-\eps,\eps) \to \mathcal{A}_{0,1}^{r-1}(W_{\bar{p}},N),
\end{split}
\end{equation}
such that $\tilde{\mathcal{G}}(s,Y_1(s),Y_2(s),Y_3(s))=(0,0)$. This implies that $Y_1(s)=0$, $Y_2(s)=0$, $Y_3(0)=0$ and that
\[
\mathcal{F}(sV+Y_3(s))=0.
\]
Differentiating this formula at $s=0$ we obtain 
\[
D\mathcal{F}(0)\left( V+\dfrac{\partial Y_3}{\partial s} (0)\right)=0.
\]
Since $V$ is admissible we deduce 
\[
D\mathcal{F}(0) \dfrac{\partial Y_3}{\partial s} (0) =0.
\]
Since $\tfrac{\partial Y_3}{\partial s} (0)=\sum_{i=1}^{\ell} g_{h_i}X_{h_i}$, where $g_{h_i} \in C_0^{r-1}(W_{\bar{p}})$, equation \eqref{eq:adm1} implies $g_{h_i} \equiv 0$ for each $i=1,\ldots,\ell$. Therefore it follows $\tfrac{\partial Y_3}{\partial s} (0)=0$.

Hence the variation  $\Gamma_s(\bar{p})=\Gamma(s V+ Y_3(s))(\bar{p})$ coincides with $\Phi(\bar{q})$ for $s=0$ and $\bar{q}\in W_{\bar{p}}$, it has degree $d$ and its variational vector fields is given by 
\[
\dfrac{\partial \Gamma_s}{\partial s} \bigg|_{s=0}= V+ \dfrac{\partial Y_3}{\partial s} (0)=V.
\]
Moreover, $\text{supp}(Y_3) \subseteq \text{supp}(V) $. Indeed,  if $\bar{q} \notin \text{supp}(V)$,
the unique vector field $Y_3(s)$, such  $\mathcal{F}(Y_3(s))=0$, is equal to $0$ at $\bar{q}$.
\end{proof}

\begin{remark}
\label{rk:nsr}
In Proposition~\ref{prop:normaladm} we stressed the fact that a vector field $V= V^{\top}+ V^{\perp}$ is admissible if and only if $V^{\perp}$ is admissible. This follows from the additivity in $V$ of the admissibility system \eqref{System of PDEs for admissible} and the admissibility of $V^\top$. Instead of writing $V$ with respect to the adapted basis $(X_i)_i$ we consider the basis $E_1,\ldots,E_m,V_{m+1}, \ldots,V_n$ described  in Section~\ref{ss:intriniscadmsy}.

Let $A^{\perp}, B^{\perp}, C^{\perp}$ be the matrices defined in \eqref{def:albe}, $A^{\top}$ be the one described  in Remark~\ref{rk:Atop} and $A$ be the matrix with respect to the basis $(X_i)_i$ defined in \eqref{def:ab}. When we change only the basis for the vector field $V$ by \eqref{eq:tildeA} we obtain $\tilde{A}=A D_h$. Since $A^{\top}$ is the null matrix and   $\tilde{A}= (A^{\top} |\, A^{\perp})$ we conclude that $\text{rank}(A(\bar{p}))=\text{rank}(A^{\perp}(\bar{p}))$. Furthermore $\Phi$ is strongly regular at $\bar{p}$ if and only if $\text{rank}(A^{\perp}(\bar{p}))= \ell \le k$, where $k$ is the integer defined in \ref{def:k}.
\end{remark}

\subsection{Some examples of regular submanifolds}

\mbox{}

\begin{example}
Consider a hypersurface $\Sigma$ immersed in an equiregular Carnot manifold $N$, then  we have that $\Sigma$ always has degree $d$ equal to $d_{\max}^{n-1}=Q-1$, see ~ \ref{ssc:hyp}. Therefore the dimension $\ell$, defined in Section~\ref{sc:structure}, of $\Lambda_m^d(U)_p$ is equal to zero. Thus any compactly supported vector field $V$ is admissible and integrable.  When the Carnot manifold $N$ is a contact structure $(M^{2n+1}, \mh=\text{ker}(\omega))$, see \ref{contact geometry}, the hypersurface  $\Sigma$ has always degree equal to $d_{\text{max}}^{2n}=2n+1$.\end{example}

\begin{example}
Let $(E,\mh)$ be the Carnot manifold described in Section~\ref{sc:2jetspace} where $(x,y,\theta,k) \in \rr^2 \times \mathbb{S}^1 \times \rr=E$ and the distribution $\mh$ is generated by
\[
 X_1=\co \partial_x+\si \partial_y+ k \partial_{\theta}, \quad X_2= \partial_k.
\]
Clearly $(X_1,\ldots,X_4)$ is an adapted basis for $\mh$. Moreover the others no-trivial commutators are given by 
\begin{align*}
[X_1,X_4]&=-k X_1-k^2 X_3\\
[X_3,X_4]&=X_1+k X_3.
\end{align*}
Let $\Omega \subset \rr^2$ be an open set. We consider the surface $\Sigma=\Phi(\Omega)$ where
\[
 \Phi(x,y)=(x,y,\theta(x,y),\kappa(x,y))
\]
 and such that $X_1(\theta(x,y))=\kappa(x,y)$. Therefore the $\deg(\Sigma)=4$ and its tangent vectors are given by 
\begin{align*}
\tilde{e}_1=&X_1+ X_1(\kk) X_2,\\
\tilde{e}_2=&X_4-X_4(\theta)X_3+X_4(\kk)X_2.
\end{align*}
Let $g=\escpr{\cdot,\cdot}$ be the metric that makes orthonormal the adapted basis $(X_1,\ldots, X_4)$.  
Since   $(\Lambda_2^{4}(N))^{\perp}=\text{span}\{X_3 \wedge X_4\}$  the only no-trivial coefficient $c_{1 1r}$, for $r=3,4$ are given by
\[
\escpr{X_3\wedge \tilde{e}_2, X_3 \wedge X_4}=1, \quad  \text{and} \quad \escpr{X_4\wedge \tilde{e}_2, X_3 \wedge X_4}=X_4(\theta).
\]
On the other hand $c_{1 2 h}=\escpr{\tilde{e}_1 \wedge X_k, X_3 \wedge X_4}=0 $ for each $h=1,\ldots,4$, since we can not reach the degree $5$ if one of the two vector fields in the wedge has degree one. Therefore the only  equation in \eqref{eq:local system of PDEs} is given by
\begin{equation}
\tilde{e}_1(f_3)+ X_4(\theta) \tilde{e}_1(f_4)+ \sum_{h=1}^4 \left( \escpr{X_3 \wedge X_4, \tilde{e}_1\wedge [\tilde{e}_2,X_h]+  [\tilde{e}_1,X_h] \wedge \tilde{e}_2} \right) f_h=0.
 \label{eq:adm}
\end{equation}
Since $\deg(\tilde{e}_1\wedge [\tilde{e}_2,X_h]) \le 4$  we have  $\escpr{X_3 \wedge X_4, e_1\wedge [\tilde{e}_2,X_h]}=0$ for each $h=1,\ldots,4$. Since $[u X, Y]=u[X,Y]-Y(u) X$ for each $X,Y \in \mathfrak{X}(N) $ and $u \in C^{\infty}(N)$ we have 
\begin{align*}
 [\tilde{e}_1,X_h]&=[X_1,X_h]+X_1(\kk)[X_2,X_h]-X_h(X_1(\kk))X_2\\
 &=\begin{cases} -X_1(\kk)X_3 -X_1(X_1(\kk))X_2 & h=1\\
                            X_3- X_2(X_1(\kk))X_2 & h=2\\
                            X_4 -X_3(X_1(\kk))X_2 & h=3\\
                            -\kk X_1- \kk^2 X_3 -X_4(X_1(\kk))X_2     & h=4.
 \end{cases}
\end{align*}
Thus, we deduce 
\[
 \escpr{X_3\wedge X_4,[\tilde{e}_1,X_h] \wedge \tilde{e}_2}=\begin{cases} -X_1(\kk) & h=1\\
                            1 & h=2\\
                            X_4(\theta) & h=3\\
                            -\kk^2   & h=4.
 \end{cases}
\]
Hence the equation \eqref{eq:adm} is equivalent to
\begin{equation}
\tilde{e}_1(f_3)+ X_4(\theta) \tilde{e}_1(f_4)-X_1(\kk) f_1+ f_2-X_4(\theta) f_3- \kk^2 f_4=0
\label{eq:adm1tk}
\end{equation}
Since $\iota_0(\Omega)=1$, we have $\rho=n_{1}=2$, where $\rho$ is the natural number defined in \eqref{def:rho}. In this setting the matrix $C$ is given by 
\[
C=\left(\begin{array}{cccc} 1&0& X_4(\theta)& 0 \end{array} \right),
\]
Then the matrices  $A$ and $B$ are given by
\[
A=\left(\begin{array}{cc}-X_1(\kk)& 1 \end{array} \right),
\]
\[
B=\left(\begin{array}{cc}-X_4(\theta)& -\kk^2 \end{array} \right).
\]
Since  $\text{rank}(A(x,y))=1$ and the matrix $\hat{A}(x,y)$, defined in the proof of Theorem \ref{teor:locint},  is equal to $1$ for each $(x,y) \in \Omega$ we have that $\Phi$ is strongly regular at each point $(x,y)$  in $\Omega$ and the open set $W_{(x,y)}=\Omega$. Hence by Theorem  \ref{teor:locint} each admissible vector field on $\Omega$ is integrable.

On the other hand we notice that $k= n_1-\tilde{m}_1=1$. By the Gram-Schmidt process an orthonormal basis with respect to the metric $g$ is given by 
\begin{align*}
 e_1&= \dfrac{1}{\alpha_1} ( X_1+X_1(\kk)X_2),\\
 e_2&=\frac{1}{\alpha_2}\left(X_4- X_4(\theta)X_3+ \frac{X_4(\kk)}{\alpha_1^2} (X_2- X_1(\kk)X_1)\right),\\
 v_3&=\dfrac{1}{\alpha_3}(X_3+X_4(\theta)X_4),\\
  v_4&= \frac{\alpha_3}{\alpha_2 \alpha_1} \left( ( -X_1(\kk)X_1+X_2)+  \frac{X_4(\kk)}{\alpha_3^2} (X_4(\theta) X_3- X_4) \right),
\end{align*}
where we set
\begin{align*}
 \alpha_1&=\sqrt{1+X_1(\kk)^2}, \quad 
 \alpha_3=\sqrt{1+X_4(\theta)^2} \\ \alpha_2&=\sqrt{1+X_4(\theta)^2+\frac{X_4(\kk)^2}{(1+X_1(\kk)^2)}}=\frac{\sqrt{\alpha_1^2 \alpha_3^2+ X_4(\kk)^2}}{\alpha_1}.
\end{align*}
Since it holds  
\begin{align*}
&\escpr{ v_3\wedge e_2, X_3 \wedge X_4}=\frac{\alpha_3}{\alpha_2},\\
&\escpr{ v_4\wedge e_2, X_3 \wedge X_4}=0,\\
&\escpr{ [e_1,v_3] \wedge e_2, X_3 \wedge X_4}=\frac{X_4(\theta) (1-\kk^2)}{\alpha_1 \alpha_2 \alpha_3},\\
&\escpr{ [e_1,v_4]\wedge e_2, X_3 \wedge X_4}=\frac{ \alpha_3 }{\alpha_2}\left(1 + \frac{X_4(\kk)^2}{\alpha_1^2 \alpha_3^2}\right)=\frac{\alpha_2}{\alpha_3},
\end{align*}
then a vector field $V^{\perp}= \psi_3(x,y) \,  v_3 + \psi_4(x,y) \, v_4$ normal to $\Sigma$ is admissible if and only if $\psi_3, \psi_4 \in C_0^{r} (\Omega) $ verify
\begin{equation*}
\frac{\alpha_3}{\alpha_2} e_1(\psi_3)+ \frac{X_4(\theta) (1-\kk^2)}{\alpha_1\alpha_2 \alpha_3} \, \psi_3+ \frac{\alpha_2}{\alpha_3} \, \psi_4=0.
\end{equation*}
That is equivalent to 
\begin{equation}
\bar{X}_1 (\psi_3) + b^{\perp} \, \psi_3 + a^{\perp} \, \psi_4=0,
\label{eq:admnormal}
\end{equation}
where $\bar{X}_1= \cos(\theta(x,y))\partial_x + \sin(\theta(x,y)) \partial_y $ and 
\begin{align*}
b^{\perp}&=\frac{X_4(\theta)(1-X_1(\theta)^2)}{1+X_4(\theta)^2}, \\
a^{\perp}&=\alpha_1\left( 1+ \frac{X_4(\kk)^2}{\alpha_1^2 \alpha_3^2}\right).
\end{align*}
In particular, since $a^{\perp}(x,y)>0$  we have that  $\text{rank}(a^{\perp}(x,y))=1$ for all $(x,y) \in \Omega$.
Along the integral curve $\ga'(t)=\bar{X}_1$ on $\Omega$ the equation \eqref{eq:admnormal} reads 
\[
\psi'_3(t)+ b^{\perp}(t) \psi_3(t)+a^{\perp}(t) \psi_4(t)=0,
\]
where we set $f(t)=f(\ga(t))$ for each function $f:\Omega \to \rr$.
\end{example}

\begin{remark}
\label{rk:regularGP}
Let $(N,\mh)$ be a Carnot manifold such that $\mh=\text{ker}(\theta)$ where $\theta$ is a $\rr^{n-\ell}$ one form. Following \cite{Gromov, Pansu16} we say that an  immersion  $\Phi:\bar{M} \to N$ is horizontal when  the pull-back $\Phi^{*} \theta =0$ and, given a point $p \in \Phi(\bar{M})$, the subspace $T_p M \subset \mh_p$ is regular if the map 
\begin{equation}
\label{eq:ontoP}
V \to (\iota_V d \theta)_{|T_p M}
\end{equation}
is onto for each horizontal vector $V$ on $\bar{M}$. Let $X$ be an horizontal extension of $V$ on $N$ and  $Y$ be another horizontal vector field on $N$, then 
\[
d \theta(X,Y)=X(\theta(Y))-Y(\theta(X))-\theta([X,Y])=-\theta([X,Y])
\]
Assume that the local  frame $E_1,\ldots,E_m$ generate $T_p M$ at $p$ then the map \eqref{eq:ontoP} is given by
$
\theta([X,E_j] (p)),
$
for each $j=1,\ldots,m$. In \cite[Section 3]{MR4118581} the author notice that there exist special coordinates adjusted to the admissibility system such that the entries of the control matrix $A$ are $a_{ijh}=\escpr{V_i, [E_j,V_h]}$, where $V_{m+1},\ldots,V_n$ are vector fields in the normal bundle.
In this notation the surjectivity of this map coincides with the pointwise condition of maximal rank of the matrix $(a_{ijh})$. Since by equation \eqref{eq:tilde}  the rank of $A$ is independent of the metric $g$ we deduce that this regularity notion introduced by \cite{Gromov, Gromov86} is equivalent to strongly regularity at $\bar{p}$ (Definition \ref{def:regular}) for the class of horizontal immersions.
\end{remark}

\subsection{An isolated plane in the Engel group}
\begin{definition}
We say that an immersion $\Phi: \bar{M} \to N$ in an equiregular graded manifold $(N,\mh^1 \subset \ldots \subset \mh^s)$ is \emph{isolated} if the only admissible variation normal to $M=\Phi(\bar{M})$ is the trivial one.
\end{definition}

Here we provide an example of isolated surface immersed in the Engel group.

\begin{example}
\label{ex:isolated}
Let $N=\rr^4$ and $\mh=\text{span}\{X_1,X_2\}$, where 
\[
 X_1=\partial_{x_1}, \quad X_2=\partial_{x_2}+ x_1\partial_{x_3}+x_3 \partial_{x_4}
\]
and $X_3=\partial_{x_3}$ and $X_4=\partial_{x_4}$. We denote by $\mathbb{E}^4$ the Engel group given by $(\rr^4,\mh)$. Let $\Upsilon:  \Omega \subset \rr^2 \to \mathbb{E}^4$ be the immersion given by 
\[
 \Upsilon(v,\omega)=(v,0,\omega,0).
\]
Since $\Upsilon_v \wedge \Upsilon_w=X_1 \wedge X_3$ the degree $\deg(\Sigma)=3$, where  $\Sigma=\Upsilon(\Omega)$ is a plane. An admissible vector field $V=\sum_{k=1}^4 f_k X_k$ verifies the system \eqref{eq:local system of PDEs} that is given by
\begin{equation}
\label{eq:admisyst}
\begin{aligned}
 \sum_{h=1}^4 &  \ \dfrac{\partial f_h}{\partial x_1} \escpr{X_h \wedge X_3, X_{J_i}}+ \dfrac{\partial f_h}{\partial x_3}\escpr{X_1 \wedge X_h, X_{J_i}}+\\
 &+f_h \left( \escpr{[X_1,X_h] \wedge X_3,X_{J_i}}+\escpr{X_1 \wedge [X_3,X_h],X_{J_i}} \right)=0,
\end{aligned}
\end{equation}
for $X_{J_1}=X_1 \wedge X_4$, $X_{J_2}=X_2 \wedge X_4$ and $X_{J_3}=X_3 \wedge X_4$. Therefore \eqref{eq:admisyst} is equivalent to 
\[
 \begin{cases}
  &\dfrac{\partial f_4}{\partial x_3}+ f_2=0\\
  & 0=0\\
  &-\dfrac{\partial f_4}{\partial x_1}=0.
 \end{cases}
\]
Let $K=\text{supp} (V)$. First of all we have $\frac{\partial f_4}{\partial x_1}=0$. Since  $f_4 \in C^{\infty}(\Omega)$ there follows 
\[
 \dfrac{\partial f_2}{\partial x_1}=-\dfrac{\partial^2 f_4}{\partial x_3 \partial x_1}=0.
\]
Then let $(x_1,x_2) \in K$ we consider the curve 
$$\ga:s \mapsto (x_1+s,x_3)$$
 along which $f_4$ and $f_2$ are constant. Since $f_4$ and $f_2$ are compactly supported at the end point, $(x_1+s_0,x_3) \in \partial K$ we have  $f_4(x_1+s_0,x_3)=f_2(x_1+s_0,x_3)=0$. Therefore we gain $f_4=f_2\equiv0$.
Therefore the only admissible vector fields $f_1X_1+ f_3 X_3$ are tangent to $\Sigma$. Assume that there exists an admissible variation $\Gamma_s$ for $\Upsilon$, then its associated variational vector field is admissible. However we proved that the only admissible vector fields are tangent to $\Sigma$, therefore the admissible variation $\Gamma_s$ has to be tangent to $\Sigma$ and the only normal one a trivial variation, hence we conclude that the plane $\Sigma$ is isolated.

Moreover, we have that $k=1$ and the matrix $A^{\perp}$ defined in \ref{def:regular} is given by 
\[
A(u,w)=\left( 
\begin{array}{cc}
 -1 \\
0 \\
0 \\
\end{array} \right).
\]
Since $\text{rank}(A)=1< 3$ we deduce that $\Upsilon$ is not strongly regular at any point in $\Omega$.
\end{example}

In analogy with the rigidity result by \cite{MR1240644}, here we prove that $\Sigma$ is isolated without using the admissibility system. This also implies that the plane $\Sigma$ is rigid in the $C^1$ topology.

\begin{proposition}
\label{prop:rigid}
 Let $\mathbb{E}^4$ be the Engel group given by $(\rr^4,\mh)$,  where the distribution $\mh$ is generated by
\[
 X_1=\partial_{x_1}, \quad X_2=\partial_{x_2}+ x_1\partial_{x_3}+x_3 \partial_{x_4}.
\]
Let $\Omega \subset \rr^2$ be a bounded open set. Then the immersion  $\Upsilon: \Omega \to \mathbb{E}^4$ of degree $3$  given by 
\[
 \Upsilon(v,w)=(v,0,w,0)
\]
is isolated.
\end{proposition}
\begin{proof}
An admissible normal variation $\Gamma_s$ of $\Upsilon$ has to have the same degree of $\Upsilon$ and has to share the same boundary  $\Upsilon(\partial \Omega)= \partial \Sigma$, where clearly $\Sigma=\Upsilon(\Omega)$. For a fix $s$, we can parametrize  $\Gamma_s$ by
\[
\Phi: \Omega \to  \mathbb{E}^4, \quad \Phi(v,w)=(v,\phi(v,w),w, \psi(v,w)),
\]
where $\phi,\psi \in C_0^1(\Omega,\rr)$. Since $\deg(\Phi(\Omega))=3$ we gain 
\begin{equation}
\label{eq:system}
\begin{cases}
\escpr{\Phi_v \wedge \Phi_w, X_1 \wedge X_4}=0\\
\escpr{\Phi_v \wedge \Phi_w, X_2 \wedge X_4}=0\\
\escpr{\Phi_v \wedge \Phi_w, X_3 \wedge X_4}=0,\\
\end{cases}
\end{equation}
where 
\[
\Phi_v=\partial_{1}+ \phi_v \partial_2 + \psi_v \partial_4= X_1 + \phi_v (X_2-v X_3+wX_4)+ \psi_v X_4
\]
and
\[
\Phi_w= \phi_w \partial_2 +\partial_3 +\psi_w \partial_4=\phi_w(X_2-vX_3+wX_4)+X_3+\psi_w X_4.
\]
Denoting by $\pi_{4}$ the projection over the $2$-vectors of degree larger  than $3$, we have 
\begin{align*}
\pi_4(\Phi_v \wedge \Phi_w)=&(\psi_w+w \phi_w) X_1 \wedge X_4+ \phi_v(\psi_w+w \phi_w) X_2 \wedge X_4\\
&-v \phi_v(\psi_w+w \phi_w) X_3 \wedge X_4 
+ \phi_w (\psi_v+w \phi_v) X_4\wedge X_2 \\
&+ (1-v\phi_w)(\psi_v+w \phi_v) X_4\wedge X_3.
\end{align*}
Therefore \eqref{eq:system} is equivalent to 
\begin{equation}
\label{eq:system2}
\begin{cases}
\psi_w+w \phi_w=0\\
\phi_v \psi_w - \psi_v \phi_w=0\\
v(\phi_v \psi_w - \psi_v \phi_w)- (\psi_v+w\phi_v)=0.\\
\end{cases}
\end{equation}
The second equation implies that \eqref{eq:system2} is equivalent to 
\begin{equation}
\label{eq:system3}
\begin{cases}
\psi_w+w \phi_w=0\\
\phi_v \psi_w - \psi_v \phi_w=0\\
\psi_v+w\phi_v=0.\\
\end{cases}
\end{equation}
Then we notice that the first and the third equations implies the second one as it follows 
\[
\phi_v \psi_w - \psi_v \phi_w= -\phi_v w\phi_w+w \phi_v \phi_w=0.
\]
Therefore the immersion $\Phi$ has degree three if and only if 
\begin{equation}
\begin{cases}
\psi_w=-w \phi_w\\
\psi_v=-w\phi_v.\\
\end{cases}
\end{equation}
Only when the compatibility conditions (\cite[Eq. (1.4), Chapter VI]{Hartman}) for linear system of first order are given we have a solution of this system. However the compatibility condition is given by 
\[
0=\psi_{w v}-\psi_{vw}= \phi_v
\]
Since $\phi \in C_0^1(\Omega)$ we obtain $\phi\equiv0$. Therefore also $\psi_v=0$, then $\psi\equiv0$.
Hence $\Phi=\Upsilon$.
\end{proof}

\section{First variation formula for submanifolds}
\label{sc:first}

In this section we shall compute a first variation formula for the area $A_d$ of a submanifold of degree $d$. We shall give some definitions first. Assume that $\Phi:\bar{M}\to N$ is an immersion of a smooth $m$-dimensional manifold into an $n$-dimensional equiregular graded manifold endowed with a Riemannian metric $g$. Let $\mu=\Phi^*g$. Fix $\bar{p}\in\bar{M}$ and let $p=\Phi(\bar{p})$. Take a $\mu$-orthonormal basis $(\bar{e}_1,\ldots,\bar{e}_m)$ in $T_{\bar{p}}\bar{M}$ and define $e_i:=d \Phi_{\bar{p}}(\bar{e}_i)$ for $i=1,\ldots,m$. Then the degree $d$ area density $\Theta$ is defined by
\begin{equation}
\label{eq:density}
\Theta(\bar{p}):=|(e_1\wedge\ldots\wedge e_m)_d|=\bigg(\sum_{\deg(X_J)=d} \escpr{e_1\wedge\ldots\wedge e_m,(X_J)_p}^2\bigg)^{1/2},
\end{equation}
where $(X_1, \ldots, X_n)$ is an orthonormal  adapted basis of $TN$. Then we have
\[
A_d(M)=\int_{\bar{M}}\Theta(\bar{p})d\mu(\bar{p}).
\]

Assume now that $V\in\mathfrak{X}(\bar{M},N)$, then we set
\begin{equation}
\label{eq:degddiv}
(\divv_{\bar{M}}^d V)(\bar{p}):=\sum_{i=1}^m \escpr{e_1\wedge\ldots\wedge\nabla_{e_i}V\wedge\ldots\wedge e_m, (e_1\wedge\ldots\wedge e_m)_d}.
\end{equation}

Finally, define the linear function $f$ by
\begin{equation}
\label{eq:deff}
f(V_{\bar{p}}):=\sum_{\deg(X_J)=d}\escpr{e_1\wedge\ldots\wedge e_m,\nabla_{V_{\bar{p}}} X_J}\escpr{e_1\wedge\ldots\wedge e_m,(X_J)_{\bar{p}}}.
\end{equation}

Then we have the following result

\begin{theorem}
\label{thm:1st}
Let $\Phi:\bar{M}\to N$ be an immersion of degree $d$ of a smooth $m$-dimensional manifold into an equiregular graded manifold equipped with a Riemannian metric $g$. Assume that there exists an admissible variation $\Gamma:\bar{M}\times (-\eps,\eps)\to N$ with associated variational field $V$ with compact support. Then
\begin{equation}
\label{eq:1stvar}
\frac{d}{dt}\bigg|_{t=0} A_d(\Gamma_t(\bar{M}))=\int_{\bar{M}} \frac{1}{\Theta(\bar{p})}\,\big((\divv_{\bar{M}}^d V)(\bar{p})+f(V_{\bar{p}})\big) d\mu(\bar{p}).
\end{equation}
\end{theorem}

\begin{proof}
Fix a point $\bar{p}\in\bar{M}$. Clearly, $\mathcal{E}_i(t,\bar{p})=d \Gamma_{(\bar{p},t)}(\bar{e}_i)$, $i=1,\ldots,m$, are vector fields along the curve $t\mapsto \Gamma(\bar{p},t)$. Therefore, the first variation is given by
\begin{align*}
\frac{d}{dt}\bigg|_{t=0} A(\Gamma_t(\bar{M}))&=\int_{\bar{M}} \frac{d}{dt}\bigg|_{t=0} 
 |\left( \mathcal{E}_1(t)\wedge \ldots \wedge \mathcal{E}_m(t) \right)_d| d\mu(\bar{p}) \\
 &=\int_{\bar{M}} \frac{d}{dt}\bigg|_{t=0} 
 \Bigg( \sum_{\text{deg}(X_J)=d} \langle \mathcal{E}_1(t)\wedge \ldots \wedge \mathcal{E}_m(t), X_J\rangle^2 \Bigg)^{\frac{1}{2}} d\mu(\bar{p}).
\end{align*}
The derivative of the last integrand is given by
\begin{multline*}
\dfrac{1}{| (e_1\wedge \ldots \wedge e_m)_d|}
  \sum_{\text{deg}(X_J)=d} \langle e_1\wedge \ldots \wedge e_m, (X_J)_p\rangle\ \times \\
\times \Bigg(\langle e_1\wedge \ldots \wedge e_m, \nabla_{V_{\bar{p}}} X_J\rangle + \sum_{i=1}^m \langle e_1\wedge \ldots \wedge \nabla_{e_i} V \wedge \ldots \wedge e_m ,  (X_J)_p\rangle  \Bigg).
\end{multline*}
Using \eqref{eq:degddiv} and \eqref{eq:deff} we obtain \eqref{eq:1stvar}.
\end{proof}

\begin{definition}
Let $\Phi:\bar{M}\to N$ be an immersion of degree $d$ of a smooth $m$-dimensional manifold into an equiregular graded manifold equipped with a Riemannian metric $g$. We say that $\Phi$ is $A_d$-stationary, or simply stationary, if it is a critical point of the area $A_d$ for any admissible variation.
\end{definition}

\begin{proposition}
\label{pr:tangentialvariation}
Let $\Phi:\bar{M}\to N$ be an immersion of degree $d$ of a smooth $m$-dimensional manifold into an equiregular graded manifold equipped with a Riemannian metric $g$. Let $\Gamma_t$ be admissible variation whose variational field $V=V^{\top}$ is compactly supported and tangent to $M=\Phi(\bar{M})$. Then we have 
\[
\frac{d}{dt}\bigg|_{t=0} A_d(\Gamma_t(\bar{M}))=0.
\]
\end{proposition}

\begin{proof}
Since $\Gamma_t(\bar{M})\subset\Phi(M)$ for all $t$, the vector field $\bar{V}_p=d\Phi_{\bar{p}}^{-1}(V_{\bar{p}})$ is tangent to $\bar{M}$ and we have
\begin{equation*}
\frac{d}{dt}\bigg|_{t=0}A_d(M)=\int_{\bar{M}} (\bar{V}(\Theta)+\Theta\divv_{\bar{M}}\bar{V})\,d\mu=\int_{\bar{M}}\divv_{\bar{M}}(\Theta \bar{V})\,d\mu=0.
\end{equation*}
\end{proof}

\begin{lemma}
\label{lem:div}
 Let $f,g\in C^\infty(M)$and $X$ be a tangential vector field in $C^{\infty}(M,TM)$. Then there holds,
 \begin{enumerate}
  \item [(i)] $f\divv_M(X)+ X(f)=\divv_M(fX) $,
  \item[(ii)] $g  X(f)=\divv_M(f  g X)-g   f\divv_M(X)-f  X(g)$.
 \end{enumerate}

 \begin{proof}
  By the definition of divergence we obtain (i)  as follows
  \[
   \divv_M(f X)= \sum_{i=1}^m \escpr{\nabla_{e_i}(f   X),e_i}= \sum_{i=1}^m e_i(f) \escpr{X,e_i}+ f \escpr{\nabla_{e_i}(X),e_i}.
  \]
  To deduce (ii) we apply twice (i) as follows
  \[
   \divv_M( g f X)-f  X(g) =g\divv_M(f X)= g X(f)+ g  f \divv_M(X).\qedhere
  \]
 \end{proof}
 \label{lm:div}
\end{lemma}

\begin{theorem}
\label{th:fvf}
 Let $\Phi:\bar{M}\to N$ be an immersion of degree $d$ of a smooth $m$-dimensional manifold into an equiregular graded manifold equipped with a Riemannian metric $g$. Assume that there exists an admissible variation $\Gamma:\bar{M}\times (-\eps,\eps)\to N$ with associated variational field $V$ with compact support. Then
\begin{equation}
\frac{d}{dt}\bigg|_{t=0} A_d(\Gamma_t(\bar{M}))=\int_{\bar{M}} \escpr{V, \mathbf{H}_d}d\mu,
\end{equation}
where $\mathbf{H}_d$ is the vector field 
\begin{equation}
\label{eq:submeancurvature2}
 \begin{split}
&-\sum_{j=m+1}^n \sum_{i=1}^m\divv_M\big(\xi_{ij}E_i\big)N_j.
\\
& +\sum_{j=m+1}^n \sum_{i=1}^m \escpr{E_1\wedge\ldots \wedge \nabla_{E_i}N_j\wedge\ldots\wedge E_m,\frac{(E_1\wedge\ldots\wedge E_m)_d}{|(E_1\wedge\ldots\wedge E_m)_d|}}\,N_j
\\
& +\sum_{j=m+1}^n \frac{f(N_j)}{\Theta}N_j.
\end{split}
\end{equation}
In this formula, $(E_i)_i$ is a local orthonormal basis of $TM$ and $(N_j)_j$ a local orthonormal basis of $TM^\perp$. The functions $\xi_{ij}$ are given by
\begin{equation}
\label{eq:xiij}
\xi_{ij}=\escpr{E_1\wedge\ldots\wedge\stackrel{(i)}{N_j}\wedge\ldots \wedge E_m,\frac{(E_1\wedge\ldots\wedge E_m)_d}{|(E_1\wedge\ldots\wedge E_m)_d|}}.
\end{equation}
\end{theorem}

\begin{proof}
Since our computations are local and immersions are local embeddings, we shall identify locally $\bar{M}$ and $M$ to simplify the notation.

We decompose $V= V^\top+V^{\perp}$ in its tangential $ V^\top$ and perpendicular $V^{\perp}$ parts. Since $\divv_{\bar{M}}^d$ and the functional $f$ defined in \eqref{eq:deff} are additive, we use the first variation formula \eqref{eq:1stvar} and Proposition~\ref{pr:tangentialvariation} to obtain
\begin{equation*}
\frac{d}{dt}\bigg|_{t=0} A_d(\Gamma_t(\bar{M}))=\int_{\bar{M}} \frac{1}{\Theta(\bar{p})}\,\big((\divv_{\bar{M}}^d V^{\perp})(\bar{p})+f(V^{\perp}_{\bar{p}})\big) d\mu(\bar{p}).
\end{equation*}

To compute this integrand we consider a local orthonormal basis $(E_i)_i$ in $TM$ around $p$ and a local orthonormal basis $(N_j)_j$ of $TM^\perp$ with $(N_j)_j$. We have
\[
V^\perp=\sum_{j=m+1}^n\escpr{V,N_j}N_j.
\]

We compute first 
\[
\frac{\divv_{\bar{M}}^d V^\perp}{\Theta}=\sum_{i=1}^m\escpr{E_1\wedge\ldots\wedge\nabla_{E_i}V^\perp\wedge\ldots \wedge E_m,\frac{(E_1\wedge\ldots\wedge E_m)_d}{|(E_1\wedge\ldots\wedge E_m)_d|}}
\]
as
\[
\sum_{i=1}^m\sum_{j=m+1}^n \escpr{E_1\wedge\ldots\wedge\big(\nabla_{E_i}\escpr{V,N_j}N_j\big)\wedge\ldots \wedge E_m,\frac{(E_1\wedge\ldots\wedge E_m)_d}{|(E_1\wedge\ldots\wedge E_m)_d|}},
\]
that it is equal to
\begin{equation}
\label{eq:H1+H2}
\begin{split}
\sum_{i=1}^m\sum_{j=m+1}^n  \bigg(&E_i\big(\escpr{V,N_j}\big) \escpr{E_1\wedge\ldots\wedge\stackrel{(i)}{N_j}\wedge\ldots \wedge E_m,\frac{(E_1\wedge\ldots\wedge E_m)_d}{|(E_1\wedge\ldots\wedge E_m)_d|}}
\\
&+\escpr{V,N_j}\escpr{E_1\wedge\ldots \wedge\stackrel{(i)}{\nabla_{E_i}N_j}\wedge\ldots\wedge E_m,\frac{(E_1\wedge\ldots\wedge E_m)_d}{|(E_1\wedge\ldots\wedge E_m)_d|}}\bigg).
\end{split}
\end{equation}
The group of summands in the second line of \eqref{eq:H1+H2} is equal to $\escpr{V,\mathbf{H}_2}$, where
\[
\mathbf{H}_2=\sum_{i=1}^m\sum_{j=m+1}^n \escpr{E_1\wedge\ldots \wedge \stackrel{(i)}{\nabla_{E_i}N_j}\wedge\ldots\wedge E_m,\frac{(E_1\wedge\ldots\wedge E_m)_d}{|(E_1\wedge\ldots\wedge E_m)_d|}}\,N_j. 
\]
To treat the group of summands in the first line of \eqref{eq:H1+H2} we use (ii) in Lemma~\ref{lem:div}. recalling \eqref{eq:xiij} we have
\[
E_i\big(\escpr{V,N_j}\big)\xi_{ij}=\divv_M\big(\escpr{V,N_j}\xi_{ij}E_i\big)-\escpr{V,\divv_M\big(\xi_{ij}E_i\big)N_j},
\]
so that applying the Divergence Theorem we have that the integral in $M$ of the first group of summands in \eqref{eq:H1+H2} is equal to
\[
\int_M \escpr{V,\mathbf{H}_1}d\mu,
\]
where
\[
\mathbf{H}_1=-\sum_{i=1}^m\sum_{j=m+1}^n \divv_M\big(\xi_{ij}E_i\big)N_j.
\]

We treat finally the summand
\[
\frac{f(V^\bot)}{\Theta}=\sum_{i=m+1}^n\escpr{V,N_j}\frac{f(N_j)}{\Theta}=\escpr{V,\mathbf{H}_3},
\]
where
\[
\mathbf{H}_3=\sum_{j=m+1}^n \frac{f(N_j)}{\Theta}N_j.
\]
This implies the result since $\mathbf{H}_d=\mathbf{H}_1+\mathbf{H}_2+\mathbf{H}_3$.
\end{proof}

In the following result we obtain a slightly different expression for the mean curvature $\mathbf{H}_d$ in terms of Lie brackets. This expression is sometimes more suitable for computations.

\begin{corollary}
Let $\Phi:\bar{M}\to N$ be an immersion of degree $d$ of a smooth $m$-dimensional manifold into an equiregular graded manifold equipped with a Riemannian metric $g$, $M=\Phi(\bar{M})$. We consider an extension $(E_i)_i$ of a  local orthonormal basis of $TM$ and respectively an extension $(N_j)_j$ of a local orthonormal basis of $TM^\perp$ to an open neighborhood of $N$. Then the  vector field $\mathbf{H}_d$ defined in \eqref{eq:submeancurvature2} is equal to
\begin{equation}
\label{eq:submeancurvature}
 \begin{aligned}
\mathbf{H}_d=\sum_{j=m+1}^n \Big( &\divv_M \Big( \Theta  N_j-\sum_{i=1}^m \xi_{ij} E_i \Big)+ \\
& + N_j(\Theta) + \sum_{i=1}^m  \sum_{k=m+1}^n \xi_{ik} \escpr{[ E_i, N_j], N_k} \Big) N_j,\\
\end{aligned}
\end{equation}
where $\xi_{ij}$ is defined in \eqref{eq:xiij}.
\end{corollary}
\begin{proof}
Keeping the notation used in the proof of Theorem~\ref{th:fvf} we consider 
\[
\mathbf{H}_2= \sum_{i=1}^m\sum_{j=m+1}^n \escpr{E_1\wedge\ldots \wedge \stackrel{(i)}{\nabla_{E_i}N_j}\wedge\ldots\wedge E_m,\frac{(E_1\wedge\ldots\wedge E_m)_d}{|(E_1\wedge\ldots\wedge E_m)_d|}}\,N_j.
\]
Writing 
\begin{equation}
\label{eq:nablaEN}
\nabla_{E_i} N_j= \sum_{\nu=1}^m  \escpr{\nabla_{E_i} N_j, E_{\nu}} E_{\nu} + \sum_{k=m+1}^{m} \escpr{\nabla_{E_i} N_j, N_k} N_k,
\end{equation}
we gain
\[
\mathbf{H}_2= \sum_{j=m+1}^n \Big( \divv_M ( N_j ) \, |(E_1\wedge\ldots\wedge E_m)_d| + \sum_{i=1}^m\sum_{k=m+1}^n  \xi_{ik}\escpr{\nabla_{E_i} N_j, N_k} \Big) N_j.
\]
Let us consider
\begin{equation}
\label{eq:H3}
\mathbf{H}_3=\sum_{j=m+1}^n \sum_{\deg(X_J)=d}\bigg(\escpr{E_1\wedge\ldots\wedge E_m,\nabla_{N_j}X_J}\frac{\escpr{E_1\wedge\ldots\wedge E_m,X_J}}{|(E_1\wedge\ldots\wedge E_m)_d|} \bigg)\,N_j.
\end{equation}
Since the Levi-Civita connection preserves the metric, we have
\begin{equation}
\label{eq:LCmp}
\escpr{E_1\wedge\ldots\wedge E_m,\nabla_{N_j}X_J}= N_j(\escpr{E_1\wedge\ldots\wedge E_m,X_J})-\escpr{\nabla_{N_j} (E_1 \wdw E_m), X_J}.
\end{equation}
Putting the first term of the right hand side of \eqref{eq:LCmp} in \eqref{eq:H3} we obtain
\[
\sum_{\deg(X_J)=d} N_j(\escpr{E_1\wedge\ldots\wedge E_m,X_J}) \frac{\escpr{E_1\wedge\ldots\wedge E_m,X_J}}{|(E_1\wedge\ldots\wedge E_m)_d|}= N_j(\Theta).
\]
On the other hand writing 
\[
\nabla_{N_j} E_i= \sum_{\nu=1}^m  \escpr{\nabla_{N_j} E_i, E_{\nu}} E_{\nu} + \sum_{k=m+1}^{m} \escpr{\nabla_{N_j} E_i, N_k} N_k
\]
we deduce
\begin{align*}
&\sum_{i=1}^m \sum_{\deg(X_J)=d}  \escpr{E_1\wedge\ldots \wedge \stackrel{(i)}{\nabla_{N_j} E_i}\wedge\ldots\wedge E_m, X_J} \frac{\escpr{E_1\wedge\ldots\wedge E_m,X_J}}{|(E_1\wedge\ldots\wedge E_m)_d|}=\\
&= \sum_{i=1}^m \sum_{k=m+1}^n \escpr{\nabla_{N_j} E_i, N_k} \xi_{ik}.
\end{align*}
Therefore we obtain 
\[
\mathbf{H}_3=\sum_{j=m+1}^n \Big( N_j (\Theta)- \sum_{i=1}^m \sum_{k=m+1}^n \escpr{\nabla_{N_j} E_i, N_k} \xi_{ik} \Big) N_j.
\]
Since the Levi-Civita connection is torsion-free we have 
\[
\mathbf{H}_2+\mathbf{H}_3=\sum_{j=m+1}^n \Big( \divv_M ( N_j ) \, \Theta + N_j(\Theta) + \sum_{i=1}^m \sum_{k=m+1}^n \xi_{ik} \escpr{[E_i,N_j],N_k} \Big).
\]
Since $\divv_M ( N_j ) \, \Theta= \divv_M (  \Theta  \, N_j ) $ we conclude that $\mathbf{H}_d=\mathbf{H}_1+ \mathbf{H}_2+ \mathbf{H}_3$ is equal to \eqref{eq:submeancurvature}.
\end{proof}

\subsection{First variation formula for strongly regular submanifolds}
\begin{definition}
\label{def:Hcomp}
Let $\Phi: \bar{M} \to N$ be a strongly regular immersion (see \S~ \ref{sc:integrability}) at $\bar{p}$, $v_{m+1},\ldots, v_n$ be an orthonormal adapted basis of the normal bundle and $k$ be the integer defined in \ref{def:k}. Let $N_{m+1},\ldots,N_n$ be a local adapted frame of the normal bundle so that $(N_j)_p=v_j$. By Remark~\ref{rk:nsr} the immersion $\Phi$ is strongly regular at $\bar{p}$ if and only if  $\text{rank}(A^{\perp})=\ell$. Then there exists a partition of $\{m+1,\ldots, m+k\}$ into sub-indices  $h_1<\ldots <h_\ell$ and $i_1<\ldots <i_{m+k-\ell}$ such that the matrix 
\begin{equation}
\label{eq:hatAper}
\hat{A}^{\perp} (\bar{p} )=\left(\begin{array}{ccc}
\alpha_{1 h_1} (\bar{p} )& \cdots & \alpha_{1 h_{\ell}}(\bar{p} )\\
\vdots & \ddots & \vdots\\
\alpha_{\ell  h_1}(\bar{p} )& \cdots & \alpha_{\ell  h_{\ell}}(\bar{p} )
\end{array} \right)
\end{equation}
is invertible. The mean curvature vector of degree $d$ defined in Theorem \ref{th:fvf} is given by
$$\mathbf{H}_d= \sum_{j=m+1}^{n} H_d^j  N_j.$$
Then we decompose $\mathbf{H}_d$ into the following three components 
\begin{equation}
\mathbf{H}_d^{v}=\begin{pmatrix} H_d^{m+k+1} \\ \vdots \\ H_d^n \end{pmatrix}^t  , \quad
\mathbf{H}_d^h=\begin{pmatrix} H_d^{h_1} \\ \vdots \\ H_d^{h_{\ell}} \end{pmatrix}^t, \quad \text{and} \quad  \mathbf{H}_d^\iota=\begin{pmatrix} H_d^{i_1} \\ \vdots \\ H_d^{i_{m+k-\ell}} \end{pmatrix}^t
\label{eq:Hcomp}
\end{equation}
with respect to $N_{m+1},\ldots, N_n$.
\end{definition}

\begin{theorem}
Let $\Phi: \bar{M} \to N$ be a strongly regular immersion at $\bar{p}$ in an equiregular  graded manifold. Then $\Phi(\bar{M})$ is a critical point of $A_d$ if and only if the immersion $\Phi$ verifies 
\begin{equation}
\label{eq:meancruv1}
\mathbf{H}_d^{\iota}- \mathbf{H}_d^{h} (\hat{A}^{\perp})^{-1} \tilde{A}^{\perp} =0,
\end{equation}
and 
\begin{equation}
\label{eq:meancruv2}
\mathbf{H}_d^v- \mathbf{H}_d^{h}(\hat{A}^{\perp})^{-1} B^{\perp} -\sum_{j=1}^m E_j^{*}\left( \mathbf{H}_d^{h} \,(\hat{A}^{\perp})^{-1} C_j^{\perp}  \right) =0,
\end{equation}
where  $E_j^{*}$ is the adjoint operator of $E_j$ for $j=1,\ldots,m$ and $\mathbf{H}_d^v$, $\mathbf{H}_d^h$ and $\mathbf{H}_d^{\iota}$ are defined in \eqref{eq:Hcomp}, $B^{\perp}$, $C_j^{\perp}$ in \ref{ss:intriniscadmsy}, $\hat{A}^{\perp}$ in \eqref{eq:hatAper} and $\tilde{A}^{\perp}$ is the $\ell \times (m+k-\ell)$ matrix given by the columns $i_1, \ldots, i_{m+k-\ell}$ of $A^{\perp}$. 
\end{theorem}
\begin{proof}
Since $\Phi: \bar{M} \to N$ is a normal strongly regular immersion then by Theorem~\ref{teor:locint} each normal admissible vector field \[
V^{\perp}=\sum_{i=m+1}^{m+k} \phi_i \, N_i + \sum_{r=m+k+1}^n \psi_r \, N_r\] is integrable. Keeping in mind the sub-indices in Definition \ref{def:Hcomp},  we set 
\begin{equation}
\Psi=\begin{pmatrix} \psi_{m+k+1} \\ \vdots \\ \psi_n \end{pmatrix}  , \quad
\Gamma=\begin{pmatrix} \phi_{h_1} \\ \vdots \\ \phi_{h_{\ell}} \end{pmatrix} \quad \text{and} \quad \Upsilon=\begin{pmatrix} \phi_{i_1} \\ \vdots \\ \phi_{i_{m+k-\ell}} \end{pmatrix}.
\label{eq:PGU}
\end{equation}
Since the immersion $\Phi: \bar{M}\to N$ is  strongly regular, the admissibility condition \eqref{eq:normal system of PDEs3} for $V^{\perp}$ is equivalent to  
\begin{equation}
\Gamma= -(\hat{A}^{\perp})^{-1} 
\bigg( \sum_{j=1}^m C^{\perp}_j \, E_j (\Psi) + B^{\perp} \Psi +\tilde{A}^{\perp}\Upsilon \bigg).
\label{eq:normadm2}
\end{equation}
By Theorem~\ref{th:fvf} the first variational formula is given by 
\begin{equation*}
\begin{aligned}
\frac{d}{dt}\bigg|_{t=0} A_d(&\Gamma_t(\bar{M}))=\int_{\bar{M}} \escpr{V^{\perp}, \mathbf{H}_d}\\
&=\int_{\bar{M}}  \mathbf{H}_d^{v} \,  \Psi +   \mathbf{H}_d^{\iota} \, \Upsilon +  \mathbf{H}_d^{h}  \Gamma\\
&=\int_{\bar{M}}  \mathbf{H}_d^{v} \,  \Psi +   \mathbf{H}_d^{\iota} \, \Upsilon -  \mathbf{H}_d^{h} \,(\hat{A}^{\perp})^{-1} 
\bigg( \sum_{j=1}^m C^{\perp}_j \, E_j (\Psi) + B^{\perp} \Psi +\tilde{A}^{\perp}\Upsilon \bigg) \\
&=\int_{\bar{M}}  \bigg( \mathbf{H}_d^{\iota}- \mathbf{H}_d^{h} (\hat{A}^{\perp})^{-1} \tilde{A}^{\perp}  \bigg) \Upsilon+\\
& \qquad + \int_{\bar{M}} \bigg(\mathbf{H}_d^v- \mathbf{H}_d^{h}(\hat{A}^{\perp})^{-1} B^{\perp} -\sum_{j=1}^m E_j^{*}\bigg( \mathbf{H}_d^{h} \,(\hat{A}^{\perp})^{-1} C_j^{\perp}  \bigg)  \bigg) \Psi,
\end{aligned}
\end{equation*}
for every $\Psi \in C_0^{\infty}(W_{\bar{p}}, \rr^{n-m-k}) , \Upsilon \in C_0^{\infty}(W_{\bar{p}}, \rr^{k-\ell})$. By the arbitrariness of $\Psi$ and $\Upsilon$, the immersion  $\Phi$ is a critical point of the area $A_d$ if and only if it satisfies equations \eqref{eq:meancruv1} and \eqref{eq:meancruv2} on $W_{\bar{p}}$.
\end{proof}

\begin{example}[First variation for a hypersurface in a contact manifold]
\label{ex:fvhcm}
 Let $(M^{2n+1}, \omega)$ be a contact manifold such that $\mh= \ker(\omega)$, see \S~ \ref{contact geometry}. Let $T$ be the Reeb vector associated to this contact geometry and $g$ the Riemannian metric on $M$ that extends a given metric on $\mh$ and makes $T$ orthonormal to $\mh$. Let $\nabla$ be  the Riemannian connection associated to $g$. 

Let us consider a hypersurface $\Sigma$ immersed in $M$. As we showed in \S~\ref{contact geometry}, the degree of $\Sigma$ is maximum and equal to $2n+1$, thus each compactly supported vector field $V$ on $\Sigma$ is admissible. Following \S~\ref{contact geometry}, we consider the unit normal $N$ to $\Sigma$ and its horizontal projection  $N_h$.  As in \S~\ref{contact geometry}, we consider the vector fields 
$
\nu_h=\frac{N_h}{|N_h|}, \qquad 
$
and  $e_1,\ldots,e_{2n-1}$ an orthonormal basis of $T_p \Sigma \cap \mh_p $.
A straightforward computation, contained in \cite{phdthesis}, shows that the mean curvature $H_d$ deduced in \eqref{eq:submeancurvature} coincide with
\begin{equation}
\label{eq:MCECM}
\begin{aligned}
 \mathbf{H}_d 
 &=-\divv_{\Sigma}^h (\nu_h) +\escpr{[\nu_h,T],T}.
\end{aligned}
\end{equation}
When $\escpr{[\nu_h,T],T}=0$  we obtain  well known horizontal divergence of the horizontal normal. This definition of mean curvature for an immersed hypersurface was first given by S.Pauls \cite{scottpauls} for graphs over the $x,y$-plane in $\mathbb{H}^1$, later extended by J.-H. Cheng, J.-F. Hwang, A. Malchiodi and P. Yang in \cite{ChengMalchiodi} in a $3$-dimensional pseudo-hermitian manifold. In a more general setting this formula was deduced in \cite{HladPauls,DanielliGarofalo}. For more details see also \cite{GRisoperegions,CapognaCittiManfre,Nataliya,Galli, RitoreRosales, RitoreRosales1}.
\end{example}

\begin{example}[First variation for ruled surfaces in an Engel Structure]
\label{ex:2jetspace}
 Here we compute the mean curvature equation for the surface  $\Sigma\subset E$ of degree $4$ introduced in Section \ref{sc:2jetspace}.
 In \eqref{adapted vector tangent to Sigma}  we determined the tangent adapted basis
 \begin{align*}
    \tilde{E}_1&=\co \Phi_x+ \si \Phi_y= X_1+X_1(\kk)X_2,\\
    \tilde{E}_2&=-\si \Phi_x+\co \Phi_y=X_4-X_4(\theta)X_3+X_4(\kk)X_2
\end{align*}
A basis for the space $(TM)^{\perp}$ is given by 
 \begin{align*}
    \tilde{N}_3&=X_4(\theta)X_4+ X_3 \\
    \tilde{N}_4&=X_1(\kk)X_1-X_2+X_4(\kk)X_4
\end{align*}
By the Gram–Schmidt process we obtain an orthonormal basis with respect to the metric $g$ as follows
\begin{align*}
 E_1&=\dfrac{\tilde{E}_1}{|\tilde{E}_1|}= \dfrac{1}{\alpha_1} ( X_1+X_1(\kk)X_2),\\
 E_2&=\frac{1}{\alpha_2}\left(X_4- X_4(\theta)X_3+ \frac{X_4(\kk)}{\alpha_1^2} (X_2- X_1(\kk)X_1)\right)\\
 N_3&=\dfrac{1}{\alpha_3}(X_3+X_4(\theta)X_4)\\
  N_4&= \frac{\alpha_3}{\alpha_2 \alpha_1} \left( ( -X_1(\kk)X_1+X_2)+  \frac{X_4(\kk)}{\alpha_3^2} (X_4(\theta) X_3- X_4) \right)
\end{align*}
where we set
\begin{align*}
 \alpha_1&=\sqrt{1+X_1(\kk)^2}, \quad 
 \alpha_3=\sqrt{1+X_4(\theta)^2} \\ \alpha_2&=\sqrt{1+X_4(\theta)^2+\frac{X_4(\kk)^2}{(1+X_1(\kk)^2)}}=\frac{\sqrt{\alpha_1^2 \alpha_3^2+ X_4(\kk)^2}}{\alpha_1}
\end{align*}
and 
\[
N_h= -X_1(\kk)X_1+X_2, \quad \nu_h=\frac{1}{\alpha_1}(-X_1(\kk)X_1+X_2)
\]
Since the degree of $\Sigma$ is equal to $4$ we deduce that 
\[
(E_1\wedge E_2)_4= \frac{1}{\alpha_1 \alpha_2} (X_1\wedge X_4 + X_1(\kk) X_2 \wedge X_4),
\]
then it follows  $|(E_1\wedge E_2)_4|= \alpha_2^{-1}$ and
\[
\frac{(E_1\wedge E_2)_4}{|(E_1\wedge E_2)_4|}= \frac{1}{\alpha_1}(X_1\wedge X_4 + X_1(\kk) X_2 \wedge X_4).
\]
A straightforward computation shows that $\xi_{i3}$ for $i=1,2$ defined in \eqref{eq:submeancurvature} are given by
\begin{align*}
\xi_{13}&=\escpr{N_3 \wedge E_2, \frac{(E_1\wedge E_2)_4}{|(E_1\wedge E_2)_4|}}=0,\\
\xi_{23}&=\escpr{E_1 \wedge N_3, \frac{(E_1\wedge E_2)_4}{|(E_1\wedge E_2)_4|}}=\frac{X_4(\theta)}{\alpha_3},\\
\xi_{14}&=\escpr{N_4 \wedge E_2, \frac{(E_1\wedge E_2)_4}{|(E_1\wedge E_2)_4|}}=0,\\
\xi_{24}&=\escpr{E_1 \wedge N_4, \frac{(E_1\wedge E_2)_4}{|(E_1\wedge E_2)_4|}}=-\frac{X_4(\kk)}{\alpha_1\alpha_2  \alpha_3 } 
\end{align*}
Since we have  
\[
  \frac{1}{\alpha_2} N_3 - \frac{X_4(\theta)}{\alpha_3} E_2=   \frac{\alpha_3}{\alpha_2} X_3- \frac{X_4(\theta)X_4(\kk)}{\alpha_1 \alpha_2 \alpha_3} \nu_h .
\]
and 
\begin{align*}
\frac{1}{\alpha_2} N_4+ \frac{X_4(\kk)}{ \alpha_1 \alpha_2 \alpha_3} E_2
&= \frac{1}{\alpha_2^2}\bigg(\frac{\alpha_3}{ \alpha_1} \bigg(N_h+ \frac{X_4(\kk)}{\alpha_3^2}(X_4(\theta)X_3-X_4) \bigg)\\
& \qquad\qquad + \frac{X_4(\kk)}{\alpha_1 \alpha_3} \bigg( -X_4(\theta)X_3+X_4 -\frac{X_4(\kk)}{\alpha_1^2} N_h \bigg) \bigg) \\
&= \frac{1}{\alpha_2^2 \alpha_1} ( {\alpha_3}N_h + \frac{X_4(\kk)^2}{\alpha_3 \alpha_1^2} N_h)
= \frac{1}{\alpha_1 \alpha_3} N_h
\\
&=\frac{1}{\alpha_3} \nu_h
\end{align*}
it follows that the third component of $\mathbf{H}_d$ is equal to 
\begin{align*}
H_d^3=&-\divv_M \left(\frac{\alpha_3}{\alpha_2} X_3- \frac{X_4(\theta)X_4(\kk)}{\alpha_1 \alpha_2 \alpha_3} \nu_h\right) -N_3(\alpha_2^{-1}) \\
&\quad+ \frac{X_4(\theta)}{\alpha_3} \escpr{[N_3,E_2],N_3}- \frac{ X_4(\kk)}{\alpha_3 \alpha_2 \alpha_1 } \escpr{[N_3,E_2],v_4}
\end{align*}
and the fourth component of $\mathbf{H}_d$ is equal to
\begin{align*}
H_d^4=&-\divv_M \left(\frac{\nu_h}{\alpha_3}\right) -N_4(\alpha_2^{-1}) + \frac{X_4(\theta)}{\alpha_3} \escpr{[N_4,E_2],N_3}- \frac{ X_4(\kk)}{\alpha_3 \alpha_2 \alpha_1 } \escpr{[N_4,E_2],N_4}.
\end{align*}
Then first variation formula is given by 
\begin{equation}
\label{eq:fvfengel}
A_{d}(\Gamma_t(\Omega))=\int_{\Omega} \escpr{V^{\perp}, \mathbf{H}_d}= \int_{\Omega} H^3_d \, \psi_3 + H^4_d \, \psi_4
\end{equation}
for each $\psi_3, \psi_4 \in C_0^{\infty}$ satisfying \eqref{eq:admnormal}. Following Theorem~\ref{teor:locint} for each $\psi_3 \in C_0^{\infty} $ we deduce 
\begin{equation}
\label{eq:admnormal2}
\psi_4=- \frac{\bar{X}_1 (\psi_3) + b^{\perp} \psi_3}{a^{\perp}},
\end{equation}
since $a^{\perp}>0$. 
\begin{lemma}
Keeping the previous notation. Let $f,g:\Omega \rightarrow \rr$ be functions in $C_0^1(\Omega)$ and  
 \begin{align*}
 \bar{X}_1&=\cos( \theta(x,y)) \partial_x +\sin(\theta(x,y)) \partial_y, \\
 X_4&=-\sin(\theta(x,y)) \partial_x+\cos( \theta(x,y)) \partial_y
 \end{align*}
Then there holds 
 $$\int_{\Omega} g \bar{X}_1 (f) + \int_{\Omega}f g  \bar{X}_4 (\theta)=-\int_{\Omega} f \bar{X}_1 (g) .$$
 \label{lm:intbyparts}
 \end{lemma}
By Lemma~\ref{lm:intbyparts} and the admissibility equation \eqref{eq:admnormal2} we deduce that  \eqref{eq:fvfengel} is equivalent to 
\[
\int_{\Omega} \bigg(H_d^3 -\frac{b^{\perp}}{a^{\perp}}H_d^4 + \bar{X}_1\left(\frac{H_d^4}{a^{\perp}}\right) + X_4(\theta) \frac{H_d^4}{a^{\perp}} \bigg) \psi_3,
\]
for each $\psi_3 \in C_0^{\infty} (\Omega)$. Therefore a straightforward computation shows that  minimal $(\theta,\kk)$-graphs for the area functional $A_4$ verify the following third order PDE
\begin{equation}
\label{eq:EL-degree4}
\bar{X}_1(H_d^4)+ a^{\perp}  H_d^3  +  \bigg( \frac{X_4(\theta) }{\alpha_3^2}  [X_1, X_4](\theta) - \frac{1}{a^{\perp}}\bar{X}_1\left(a^{\perp} \right) \bigg)  H_d^4 =0.
\end{equation}
 \end{example}

\bibliography{degree}

\end{document}